\documentclass[11pt,oneside]{article}

\usepackage{amsthm,amsmath,amssymb,amsfonts}
\usepackage{mathtools,xfrac}
\usepackage{url,setspace,color}
\usepackage{subfigure,graphicx,epsfig,tikz,caption}
\usepackage[normalem]{ulem}
\usepackage{enumerate}
\usepackage{geometry,algorithmic,algorithm,verbatim}
\usepackage{makeidx,latexsym}
\usepackage[colorlinks=true]{hyperref}

\usepackage{authblk}
\usepackage[in]{fullpage}
\makeatletter
\def\imod#1{\allowbreak\mkern10mu({\operator@font mod}\,\,#1)}
\makeatother
\usepackage[compact]{titlesec}

\definecolor{gold}{rgb}{0.85,0.65,0}

\newcommand{\be}{\begin{eqnarray}}
\newcommand{\ee}[1]{\label{#1}\end{eqnarray}}
\newcommand{\ese}{\end{eqnarray*}}
\newcommand{\bse}{\begin{eqnarray*}}
\def\beq{\begin{equation}}
\def\eeq{\end{equation}}

\def\fnote#1{\footnote}
\newcommand{\epr}{\hfill\hbox{\hskip 4pt \vrule width 5pt height 6pt depth 1.5pt}\vspace{0.0cm}\par}

\newcommand{\grad}{\ensuremath{\nabla}}

\def\R{{\mathbb{R}}}

\def\cC{{\cal C}}

\def\cF{{\cal F}}

\def\cK{{\cal K}}

\def\cP{{\cal P}}

\def\cR{{\cal R}}
\def\cS{{\cal S}}

\newcommand{\bbR}{\mathbb{R}}

\DeclareMathOperator{\Opt}{Opt}

\DeclareMathOperator*{\argmin}{arg\,min}

\DeclareMathOperator{\Diag}{Diag}

\DeclareMathOperator{\Null}{Null}

\DeclareMathOperator{\Range}{Range}

\def\dim{\mathop{{\rm dim}\,}}

\DeclareMathOperator{\intt}{int}

\DeclareMathOperator{\rint}{rel int}

\DeclareMathOperator{\bd}{bd}

\DeclareMathOperator{\conv}{conv}

\DeclareMathOperator{\cone}{cone}
\DeclareMathOperator{\clconv}{\overline{conv}}

\DeclareMathOperator{\ccnh}{\overline{cone}}
\def\Ext{{\mathop{\rm Ext}}}

\DeclareMathOperator{\Rec}{Rec}

\DeclareMathOperator{\apex}{apex}

\def\log{\mathop{{\rm log}}}

\theoremstyle{plain}
\newtheorem{theorem}{Theorem}

\newtheorem{lemma}[theorem]{Lemma}
\newtheorem*{lemma*}{Lemma}

\newtheorem*{proposition*}{Proposition}
\newtheorem{corollary}[theorem]{Corollary}
\newtheorem{observation}[theorem]{Observation}
\newtheorem*{observation*}{Observation}
\newtheorem{example}[theorem]{Example}
\newtheorem*{example*}{Example}

\theoremstyle{definition}

\newtheorem{condition}{Condition}[section]

\theoremstyle{remark}
\newtheorem{remark}{Remark}[section]
\newtheorem*{remark*}{Remark}

\theoremstyle{claim}

\renewcommand{\epr}{\hfill$\diamondsuit$\vspace{0.25cm}\par}

\title{A Second-Order Cone Based Approach for Solving the Trust-Region Subproblem and Its Variants}

\author[1]{Nam Ho-Nguyen}
\author[1]{Fatma K{\i}l{\i}n\c{c}-Karzan}
\affil[1]{Tepper School of Business, Carnegie Mellon University, Pittsburgh, PA, 15213, USA.}
\date{Submitted on 10 March 2016; Revised on 09 October 2016}

\begin{document}

\maketitle
\begin{abstract}
We study the trust-region subproblem (TRS) of minimizing a nonconvex quadratic function over the unit ball with additional conic constraints. Despite having a nonconvex objective, it is known that the classical TRS and a number of its variants are polynomial-time solvable. In this paper, we follow a second-order cone (SOC) based approach to derive an exact convex reformulation of the TRS under a structural condition on the conic constraint. Our structural condition is immediately satisfied when there is no additional conic constraints, and it generalizes several such conditions studied in the literature. As a result, our study highlights an explicit connection between the classical nonconvex TRS and smooth convex quadratic minimization, which allows for the application of cheap iterative methods such as Nesterov's accelerated gradient descent, to the TRS. Furthermore, under slightly stronger conditions, we give a low-complexity characterization of the convex hull of the epigraph of the nonconvex quadratic function intersected with the constraints defining the domain without any additional variables. We also explore the inclusion of additional hollow constraints to the domain of the TRS, and convexification of the associated epigraph.
\end{abstract}

\section{Introduction}\label{sec:Intro}

In this paper, we study the classical \emph{trust-region subproblem} (TRS) \cite{Conn.et.al.2000} and its polynomial-time solvable variants given by
\begin{equation} \label{eqn:trsSOC-conic}
\Opt _h :=  \min_{y \in \R^{n}}
\left\{ h(y) := y^\top  Q y + 2 \, g^\top y :~
\begin{array}{rcl} \|y\| &\leq& 1 \\ Ay - b &\in& \cK  
\end{array} \right\},
\end{equation} 
where $\|y\|$ denotes the Euclidean norm of $y$, $A \in \R^{m \times n}$, $b \in \R^m$, and $\cK\subseteq\R^m$ is a closed convex cone.  
Throughout the paper, we assume that the minimum eigenvalue of $Q$ is negative, that is, $\lambda_Q := \lambda_{\min}(Q) < 0$ and the domain of the problem is nonempty. 
Problem~\eqref{eqn:trsSOC-conic} is equivalent to the \emph{classical TRS} when there are no additional conic constraints, i.e., $A = I_n$, $b = 0$, and $\cK = \R^n$. 
That is, the classical TRS is given by 
\begin{equation}\label{eqn:trsSOC-classical}
 \min_{y\in\R^n} \left\{ h(y) := y^\top Q y + 2 g^\top y :~ \|y\| \leq 1 \right\}.
\end{equation}

The classical TRS is an essential ingredient of trust-region methods that are commonly used to solve continuous nonconvex optimization problems (see \cite{Conn.et.al.2000,NocedalWright2000numerical,PongWolkowicz2014} and references therein). In each iteration of a trust-region method, a quadratic approximation of the objective function is built and then optimized over a ball, called trust region, (or intersection of a ball with linear or conic constraints originating from the original problem) to find the new search point. 
TRS and its variants are also encountered in the context of 
robust optimization under matrix norm or polyhedral uncertainty (see \cite{BenTal_ElGhaoui_Nemirovski_09,BertsimasBrownCaramanis2011} and references therein), 
nonlinear optimization problems with discrete variables \cite{BuchheimDSPP2013,BurerAnstreicher2013}, 
least-squares problems \cite{zhang2010derivative}, 
constrained eigenvalue problems \cite{gander1989constrained}, 
and more.

As stated above, the optimization problem in \eqref{eqn:trsSOC-conic} is nonlinear and nonconvex when $\lambda_Q < 0$. Nevertheless, it is well-known that the semidefinite programming (SDP) relaxation for the classical TRS is exact, and classical TRS and a number of its variants can be solved in polynomial time via SDP-based techniques  \cite{Rendl_Wolkowicz_97,FortinWolkowicz2004} or using specialized nonlinear algorithms, e.g., \cite{Gould_LRT_99,More_Sorensen_83}. 

Several variants of the classical TRS that enforce additional constraints on the trust region have been proposed. Among these the most commonly studied is the case when $\cK$ is taken to be a nonnegative orthant, i.e., the unit ball is intersected with additional linear constraints modeled via the polyhedral set $\{y\in\R^n:\;Ay-b\in\cK\}$. TRS with additional linear inequalities arises in nonlinear programming and robust optimization (see \cite{Burer2015,JeyakumarLi2013} and references therein) and is studied in \cite{BienstockMichalka2014,Burer2015,BurerAnstreicher2013,BurerYang2014,JeyakumarLi2013,SturmZhang2003,YeZhang2003} under a variety of assumptions.  Specifically, \cite{BurerAnstreicher2013,SturmZhang2003} give a tight semidefinite formulation when there is a single linear constraint $a^\top y \leq b$ based on an additional constraint derived from second-order cone (SOC) based reformulation linearization technique (SOC-RLT). 
This approach was extended to two linear constraints in \cite{BurerAnstreicher2013,YeZhang2003} and the tightness of the SDP relaxation is shown when the linear constraints are parallel.  
More recently, Burer and Yang \cite{BurerYang2014} give a tight SDP relaxation with additional SOC-RLT constraints for an arbitrary number of linear constraints, under the condition that these additional linear inequalities do not intersect on the interior of the unit ball. 
We refer the readers to Burer \cite{Burer2015} for a recent survey and related references for the results on tight SDP relaxations associated with these variants. 
Following a different approach, Bienstock and Michalka \cite{BienstockMichalka2014} show that TRS with linear inequality constraints is polynomial-time solvable under the milder condition that the number of faces of the linear constraints intersecting with the unit ball is polynomially bounded.

TRS with additional conic constraints originate when the trust-region algorithm is applied to conic constrained optimization problems with nonconvex objective. Most notable example in this context is the well-known Celis-Dennis-Tapia (CDT) problem \cite{CelisDennisTapia1985} where a nonconvex quadratic is minimized over the intersection of two-ellipsoids. See also Ben-Tal and den Hertog \cite{BenTalDenHertog2014} for several applications of the TRS with additional conic quadratic constraints arising in the context of robust quadratic programming. 
Recently, Jeyakumar and Li \cite{JeyakumarLi2013} prove convexity of the joint numerical range, exactness of the SDP relaxation, and strong Lagrangian duality for the TRS with additional linear and SOC constraints. A key tool in their analysis is to recast the TRS as a convex quadratic minimization problem under a dimensionality condition. 

Hollow constraints defined by a single ellipsoid \cite{BenTalTeboulle1996,BM14,PongWolkowicz2014,SternWolkowicz1995,YeZhang2003}, several ellipsoids \cite{BienstockMichalka2014,YangBurer2016} or arbitrary  quadratics constraints \cite{Bienstock2016} have also attracted some attention in the literature. These approaches are once again either lifted SOC-based or SDP-based convexification schemes or customized algorithms. We discuss these further in Section~\ref{sec:TRS-additional-constraints}.

While the SDP reformulations of the classical TRS and its variants can be solved using interior-point methods in polynomial time \cite{alizadeh1995interior,nesterov1994interior}, this approach is not practical because the worst-case complexity of these methods for solving SDPs is a relatively large polynomial and there exist faster methods. 
That said, the classical TRS is closely connected to eigenvalue problems. In the specific case of classical TRS where the objective is convex, i.e., when $Q$ is positive semidefinite, this problem becomes simply the minimization of a smooth convex function over the Euclidean ball, and thus it can be solved efficiently via iterative first-order methods (FOMs) such as Nesterov's accelerated gradient descent algorithm \cite{Nesterov_83}. 
Moreover, in the nonconvex case with $\lambda_Q<0$, when the problem is purely quadratic, i.e., when $g=0$ as well, the classical TRS reduces to finding the minimum eigenvalue of $Q$. This can be approximated
efficiently via the Lanczos method \cite[Chapter 10.1]{GolubVanLoan1996book} in practice. When $g \neq 0$, even though the classical TRS is no longer equivalent to an eigenvalue problem and these methods cannot be applied directly, this observation has led to the development of efficient, matrix-free algorithms that are based solely on matrix-vector products. 
The dual-based algorithms of \cite{More_Sorensen_83}, \cite{Rendl_Wolkowicz_97} and \cite{sorensen1997minimization}, the generalized Lanczos trust-region method of \cite{Gould_LRT_99}, 
 and the recent developments of \cite{AdachiIwata2015,erway2009subspace,erway2009iterative,gould2010solving,HazanKoren2016,rojas2001new} are examples of such iterative algorithms. 
More recently, for TRS with a single additional linear constraint, the papers \cite{SalahiFallahi2016,SalahiTaati2015,SalahiTaatiWolkowicz2016} explore strong Lagrangian duality, and derive numerically efficient algorithms from this.  
In most cases, these algorithms for classical TRS and its variants are presented together with their convergence proofs. 
Nevertheless, to the best of our knowledge, the theoretical runtime evaluation of these algorithms lacks formal guarantees with the exception of recent work \cite{HazanKoren2016} (done in a probabilistic fashion). 
In addition, in most of these iterative methods, numerical difficulties are reported in the so-called `hard case' \cite{More_Sorensen_83}, when the linear component vector $g$ is nearly orthogonal to the eigenspace of the smallest eigenvalue of $Q$.  
In many cases, the lack of provable worst-case convergence bounds for the classical TRS is attributed to the hard case. 
As a result, most research on specific algorithms for the classical TRS thus far focuses on addressing this issue. 

Recently, Hazan and Koren \cite{HazanKoren2016} suggested a linear-time algorithm for approximately solving the classical TRS within a given tolerance $\epsilon$ on the objective value. 
Their approach relies on an efficient, linear-time solver for a specific SDP relaxation of a feasibility version of the classical TRS and reduces the classical TRS into a series of eigenvalue computations. Specifically, they exploit the special structure of the dual problem, a one-dimensional problem for which bisection techniques can be applied, to avoid using interior-point solvers. 
Each dual step of their algorithm requires a single approximate maximal eigenvalue computation which takes $O\left( N \frac{\sqrt{\Gamma}}{\sqrt{\epsilon}} \log\left(\frac{n}{\delta}   \log\left( \Gamma/\epsilon \right) \right) \right)$ time to achieve an $\epsilon$-accurate estimate with probability at least $1-\delta/\log\left( \Gamma/\epsilon \right)$, where $N$ is the number of nonzero entries in $Q$, $\Gamma := \max\left\{2( \|Q\| + \|g\|), 1\right\}$, and $\|Q\|$ stands for the spectral norm of the matrix $Q$, i.e., the maximum absolute eigenvalue. Their overall algorithm converges in $O\left( \log\left(\frac{\Gamma}{\epsilon}\right) \right)$ iterations. Then a primal solution is recovered by solving a small linear program formed by the dual iterates. Finally, they provide an efficient and accurate rounding procedure for converting the SDP solution into a feasible solution to the classical TRS. Consequently, their approach does not require the use of interior-point SDP solvers and bypasses the difficulties noted for the hard case of the classical TRS. The overall complexity (elementary arithmetic operations) of their approach is
\[
O\left( N \frac{ \sqrt{\Gamma} \log\left( \Gamma / \epsilon \right)}{\sqrt{\epsilon}} \log\left(\frac{n}{\delta} \log\left(\frac{\Gamma}{\epsilon}\right)\right)  \right).
\]
Thus, their approach runs in time linear in the number of nonzero entries of the input and it can exploit data sparsity.

These algorithmic developments for TRS have been complemented with research on convex hull characterization of sets associated with TRS. In this respect, \cite{Burer2015} presents a nice summary of such results given for the lifted SDP representations. 
The epigraph of TRS is closely related to convex hulls of sets defined as the intersection of convex and nonconvex quadratics. Such sets cover two-term disjunctions applied to an SOC or its cross-sections arising in the context of Mixed Integer Conic Programming or reverse convex constraints based on ellipsoids, and thus have been studied under a variety of assumptions (see Burer and K{\i}l{\i}n\c{c}-Karzan \cite{BKK14} and references therein).  
In particular, 
nonconvex sets obtained from the intersection of a second-order-cone representable (SOCr) cone and a nonconvex cone defined by a single homogeneous quadratic, and possibly  an affine hyperplane were studied in \cite{BKK14}. For such sets, under several easy-to-verify conditions, \cite{BKK14} suggests a simple, computable convex relaxation where the nonconvex cone is replaced by an additional SOCr cone, and identifies several stronger conditions guaranteeing the tightness of these relaxations, in terms of giving the associated closed conic hulls and closed convex hulls of these sets.  
These conditions have been further verified in many specific cases, and it was shown in \cite{BKK14} that the classical TRS can be solved via the optimization of two SOC-based programs. 
Similar convex hull descriptions of a single SOC or its cross-section  intersected with a general nonconvex quadratic are also studied recently in \cite{MV14} under different assumptions.

In this paper, as opposed to the previous specialized algorithms or approaches that work in a lifted space, e.g., SDP-based relaxations, we follow an SOC-based approach in the original space of variables to solve the classical TRS and its variants with conic constraints  \eqref{eqn:trsSOC-conic} or hollows. That is, under easy-to-verify conditions, we derive tight SOC-based convex reformulations and convex hull characterizations of sets associated with the TRS with additional conic constraints \eqref{eqn:trsSOC-conic}. 
Our contributions can be summarized as follows.
\begin{enumerate}[(i)]
\item In Section~\ref{sec:SOCReformulation}, we study an SOC-based convex relaxation of 
\eqref{eqn:trsSOC-conic} 
in the original space of variables obtained by simply replacing the nonconvex objective function $h(y)$ in \eqref{eqn:trsSOC-conic} with the convex objective $f(y):= y^\top \left( Q - \lambda_Q I_n \right) y + 2 \, g^\top y + \lambda_Q$. We prove tightness of this relaxation under an easily checkable structural condition on the additional conic constraints $Ay-b\in\cK$ (see Theorem~\ref{thm:trsSOC-tight-convex}). For classical TRS our convex relaxation is immediately tight without any condition. In the case of nontrivial conic constraints $Ay-b\in\cK$ in \eqref{eqn:trsSOC-conic}, the conditions ensuring tightness of our convex relaxation can be somewhat stringent. We discuss these issues and relation of our condition to the existing ones from the literature in Section~\ref{sec:conditions}. 

\item
Due to the fact that our convex relaxation/reformulation works in the original space of variables and thus preserves the domain, it is immediately amenable to work with existing iterative FOMs; we discuss the associated complexity results in Section~\ref{sec:SOCComplexity}. In particular, our convex relaxation/reformulation can be built via a single minimum eigenvalue computation. In the case of classical TRS, it can then solved by minimizing a smooth convex quadratic over the unit ball via Nesterov's accelerated gradient descent algorithm \cite{Nesterov_83}.
Thus, with probability $1-\delta$, our approach solves the classical TRS to accuracy $\epsilon$ in running time
\[ O\left( N \left( \frac{ \sqrt{\|Q\|} }{\sqrt{\epsilon}} \log\left(\frac{n}{\delta}\right) + \frac{\sqrt{\|Q\|}}{\sqrt{\epsilon}} \right) \right). \]

\item Finally, in Section~\ref{sec:TRSConvexification}, we study exact and explicit SOC-based convex hull results for the epigraph of the TRS given by 
\[ 
X := \left\{ \begin{bmatrix} y\\ t \end{bmatrix}\in\R^{n+1} : \begin{array}{rcl} \|y\| &\leq& 1\\ Ay - b &\in& \cK\\ h(y) &\leq& t \end{array} \right\}.
\]
In Theorem~\ref{thm:trs-convexify}, under a slightly stronger condition, we provide an explicit characterization of convex hull of $X$ in the space of original variables. 

We also examine the inclusion of additional hollow constraints $y \in \cR = \R^n \setminus \cP$ to the TRS in Section~\ref{sec:TRS-additional-constraints}.  
In particular, these developments immediately lead to convex reformulations for several variants of TRS, including interval-bounded 
TRS (see \cite{BenTalTeboulle1996,BM14,PongWolkowicz2014,SternWolkowicz1995,YeZhang2003}), and thus have algorithmic implications. 
\end{enumerate}

From a convex reformulation perspective, the papers \cite{FortinWolkowicz2004}, \cite{JeyakumarLi2013}, \cite{BKK14}, \cite{BenTalDenHertog2014}, and \cite{Locatelli2016} are closely related to our approach.  
To handle the hard case in classical TRS, Fortin and Wolkowicz \cite{FortinWolkowicz2004} discusses a shift of the matrix $Q$, which results in the same SOC-based convex reformulation as ours.  
Nevertheless, \cite{FortinWolkowicz2004} solves the resulting problem using a modification of the SDP-based Rendl-Wolkowicz algorithm \cite{Rendl_Wolkowicz_97}. Their approach requires a case-by-case analysis to handle the hard case and lacks formal convergence guarantees. In contrast to such an approach, we propose using Nesterov's algorithm \cite{Nesterov_83}, which is not only oblivious to the hard case and thus does not requires a case-by-case analysis, but also provides formal convergence guarantees.  
Jeyakumar and Li \cite{JeyakumarLi2013} study TRS with additional linear and conic-quadratic constraints. They obtain a convex reformulation via a similar shift in the $Q$ matrix under a certain dimensionality condition on the additional constraints. We show that the conditions from \cite{JeyakumarLi2013} imply our structural condition and we provide an example where our condition is satisfied but the ones in \cite{JeyakumarLi2013} are not. 
Burer and K{\i}l{\i}n\c{c}-Karzan \cite{BKK14} also give a scheme to solve the classical TRS via SOC programming. The scheme suggested in \cite{BKK14} is in a lifted space with one additional variable and requires solving two related SOC optimization problems. In contrast, our convex reformulation is in the space of original variables and requires solving only a single minimization problem. 
Ben-Tal and den Hertog \cite{BenTalDenHertog2014} study a different SOC-based convex reformulation in a lifted space of the TRS and its variants under a simultaneously diagonalizable assumption. However, this relaxation requires a full eigenvalue decomposition of the matrix $Q$ as opposed to our relaxation which only needs a maximum eigenvalue computation. Based on the same convex reformulation as in \cite{BenTalDenHertog2014}, Locatelli~\cite{Locatelli2016} studies the TRS with additional linear constraints under a structural condition on the constraints derived from a KKT system. We show that in the case of additional linear constraints, our geometric condition is equivalent to the structural condition used in \cite{Locatelli2016} (see Lemma \ref{lem:Locatelli}).  
To the best of our knowledge, the KKT based derivations of conditions in \cite{Locatelli2016} are not extended to the conic case, yet our condition handles additional conic constraints generalizing the one from \cite{Locatelli2016} and highlights the features of underlying geometry.

On the algorithmic side, our transformation of the TRS \eqref{eqn:trsSOC-conic} is mainly based on the minimum eigenvalue of $Q$,
which can be computed to accuracy $\epsilon>0$ with probability $1-\delta$ in $O\left( N \sqrt{\|Q\|} \log(n/\delta) / \sqrt{\epsilon} \right)$ arithmetic operations using the Lanczos method (see \cite[Section 4]{KuczynskiWozniakowski1992estimating} and \cite[Section 5]{HazanKoren2016}), where $N$ is the number of nonzero entries in $Q$. 
Due to the fact that $f(y)$ is a convex quadratic function, our convex relaxation/reformulation for  \eqref{eqn:trsSOC-conic} can simply be cast as a conic optimization problem. Specifically, when there are no additional constraints, this exact convex reformulation becomes minimizing a smooth convex function over the Euclidean ball, and thus it is readily amenable to efficient FOMs. For this class of convex problems, given a desired accuracy of $\epsilon$, a classical FOM, Nesterov's accelerated gradient descent algorithm \cite{Nesterov_83}, involves only elementary operations such as addition, multiplication, and matrix-vector product computations and achieves the optimal iteration complexity of $O\left(\sqrt{\|Q\|}/\sqrt{\epsilon} \right)$. Note when the problem is convex (when $Q$ is positive semidefinite), the same complexity guarantees can be obtained by applying Nesterov's accelerated gradient descent \cite{Nesterov_83} to the problem. Thus, our approach can be seen as an analog of the latter algorithm to the general nonconvex case. 
This is the first-time that such an observation is made that the classical TRS problem can be solved by a single minimum eigenvalue computation and  Nesterov's accelerated gradient descent~\cite{Nesterov_83}. 
Moreover, our analysis highlights the connection between the TRS and eigenvalue problems, and in fact demonstrates that, up to constant factors, the complexity of solving the classical TRS is no worse than solving a minimum eigenvalue problem.  
This was empirically observed in \cite[Section 5]{Rendl_Wolkowicz_97} and our analysis provides a theoretical justification for it. 

Convexification-based approaches such as ours and \cite{BenTalDenHertog2014,BenTalTeboulle1996,HazanKoren2016,JeyakumarLi2013,Locatelli2016} work directly with convex formulations and provide a uniform treatment of the problem and thus bypass the so-called `hard case'.  
Moreover, the resulting convex formulations are then amenable to iterative FOMs from convex optimization literature which only require matrix-vector product type operations. 
To the best of our knowledge, iterative algorithms for SDP-based relaxations of the TRS have not been studied in the literature with the exception of Hazan and Koren \cite{HazanKoren2016}. 
As compared to the approach in \cite{HazanKoren2016}, we believe our approach is straightforward, easy to implement, and achieves a slightly better convergence guarantee in the worst case. In particular, our approach directly solves the TRS, as opposed to only solving a feasibility version of the TRS; thus we save an extra logarithmic factor. 
While \cite{HazanKoren2016} relies on repeatedly calling a minimum eigenvalue, our approach, as well as that of Jeyakumar and Li \cite{JeyakumarLi2013}, work with an SOC-based reformulation of the problem in the original space and requires only a single minimum eigenvalue computation. 
The convex reformulations given by Ben-Tal and Teboulle  \cite{BenTalTeboulle1996} or the one studied in Ben-Tal and den Hertog \cite{BenTalDenHertog2014} and Locatelli \cite{Locatelli2016} requires a full eigenvalue decomposition which is more expensive, 
i.e., $O(n^3)$ time. Moreover, these reformulations from \cite{BenTalDenHertog2014,BenTalTeboulle1996,Locatelli2016}  involve  additional variables and constraints, and thus FOMs applied to these  entail more complicated and expensive projection operations.  

Efficient algorithms to solve convex reformulation of TRS \eqref{eqn:trsSOC-conic} in the original space of variables is particularly advantageous in the context of solving robust convex quadratic programs (QPs). Robust convex QPs with ellipsoidal uncertainty are known to have close connections with the TRS (see \cite{BenTalDenHertog2014}). The function $f(x,u)$ underlying a robust convex quadratic constraint $\sup_{u\in U}f(x,u) \leq 0$ is convex in both the decision variable $x$ and the uncertainty $u$, highlighting the nonconvexity of the problem. Yet, a convex reformulation of such a robust constraint in the original space of variables allows us to recast it as $\sup_{u\in U}\tilde{f}(x,u) \leq 0$, where $\tilde{f}(x,u)$ is convex-concave in $x$ and $u$, demonstrating its hidden convexity. 
Recently, in \cite{Ho-NguyenKK2016RO} an efficient online iterative framework is introduced to solve robust convex optimization problems which bypasses the burden of taking robust counterparts.
When specialized to robust convex QPs, each iteration of this online framework requires handling each robust constraint independently and making a simple iteration towards solving the associated TRSs as opposed to completely solving the TRSs. Then efficient online FOMs capable of solving the TRS in the original space of variables becomes a key component of such an approach to solve robust convex QPs.

Our convex hull results on the epigraph of the TRS are inspired by the recent work of Burer and K{\i}l{\i}n\c{c}-Karzan \cite{BKK14} on convex hulls of general quadratic cones. 
While the SOC-based convex hull results in \cite{BKK14} are applicable to many problems, including the epigraph set associated with the classical TRS, we present a much more direct analysis specialized for TRS. There are two main benefits of our approach. First, the approach outlined in \cite{BKK14} for solving classical TRS requires the assumption that the optimal value is nonpositive. While this is not an issue for the classical TRS since its optimal value is always negative under the assumption of $\lambda_Q<0$, with the existence of additional constraints, this may no longer be true for \eqref{eqn:trsSOC-conic}. In contrast, our direct analysis does not rely on any nonpositivity assumptions of the objective value, and hence we are able to extend our results to include additional conic constraints. Second, our direct analysis of the TRS allows us to bypass verifying several conditions from \cite{BKK14} and to work directly with a single structural condition on additional conic constraints which is always satisfied in the case of the classical TRS. 

Several papers \cite{Beck2009,BeckEldar2006,JeyakumarLi2013} exploit convexity results on the joint numerical range of quadratic mappings to explore strong duality properties of the TRS and its variants. These convexity results are based on Yakubovich's $\cS$-lemma \cite{FradkovYakubovich1979} and Dines \cite{Dines1941}, see also the survey by P{\'o}lik and Terlaky \cite{PolikTerlaky2007} for a more detailed discussion. While these results as well as ours both analyze sets associated with the TRS, the actual sets in question are quite different. In the context of the TRS, the joint numerical range is a set of the form
\[ \left\{ [h(y);\; \|y\|^2;\; Ay-b] 
:~ y \in \R^n \right\} \subseteq \R^{m+2}.\]
Under certain conditions, this set is shown to be convex. In contrast, we study the epigraphical set $X$, 
which is nonconvex if $h(y)$ is, and we give its convex hull description in the original space of variables.


\emph{Notation}. 
We use Matlab notation to denote vectors and matrices. Given a matrix, $A\in\R^{m\times n}$, we let $\Null(A)$ and $\Range(A)$ denote its nullspace and range. Furthermore, we denote the minimum eigenvalue of a symmetric matrix $Q$ as $\lambda_Q:=\lambda_{\min}(Q)$ and we let $I_n$ be the $n\times n$ identify matrix. For a given symmetric matrix $Q$, the notation $Q \succeq 0$ ($Q \succ 0$) corresponds to the requirement that $Q$ is positive semidefinite (positive definite). 
Given a vector $\xi\in\R^n$, $\Diag(\xi)$ corresponds to an $n\times n$ diagonal matrix with its diagonal equal to $\xi$. 
For a set $S \subseteq \bbR^n$, we define $\intt(S), \rint(S), \bd(S), \Ext(S), \Rec(S)$, $\conv(S)$, $\clconv(S)$, $\cone(S)$ and $\ccnh(S)$ to be the interior, relative interior, boundary, set of extreme points, recession cone, convex hull, closed convex hull, conic hull, and closed conic hull of $S$ respectively. For a cone $\cK \subseteq \bbR^n$, we denote its dual cone by $\cK^*$.

\section{Tight Low-Complexity Convex Reformulation of the TRS}\label{sec:SOCReformulation}

In this section, we first present an exact SOC-based convex reformulation for the classical TRS and extend this reformulation to the TRS with additional conic constraints \eqref{eqn:trsSOC-conic} under an appropriate condition. 
We then compare and relate our condition to handle conic constraints to other conditions studied in the literature. 
Finally, we explore algorithmic aspects of solving 
our SOC-based reformulation.

\subsection{Convex Reformulation} 
\label{sec:TRSReformulation}

We start with the following simple observation, which we present without proof.
\begin{observation}\label{obs:boundary-nonconvex}
Let $\cC \subset \bbR^n$ be some bounded domain and $h:\cC \to \bbR$ be a (possibly nonconvex) function such that $h$ has no local minimum on $\intt(\cC)$. Then any optimal solution $y^*$ of the program
\[ \min_{y} \left\{ h(y) :~ y \in \cC \right\} \]
must be on $\bd(\cC)$.
\end{observation}

We next observe that when our domain $\cC$ is defined by (possibly nonconvex) constraints $c_j(y) \leq 0$, we can obtain relaxations of the nonconvex program in Observation~\ref{obs:boundary-nonconvex} by simply aggregating these constraints with appropriate weights.
\begin{lemma}\label{lem:nonconvex-underestimator-tightness}
Let $C \subseteq \bbR^n$ be a given set, $c_j(y):C \to \bbR$ for $j=1,\ldots,m$ be given functions. Suppose $h(y)$ is a given function and $f_j(y)$ are functions on the domain $\cC:=\{y:\; c_j(y) \leq 0,\;\forall j=1,\ldots,m\} \cap C$ such that $f_j(y)=h(y)-\alpha_j c_j(y)$ for some $\alpha_j \leq 0$. Let $F(y):=\max_{j=1,\ldots,m}f_j(y)$. Then 
\[
\Opt_h:=\min_y\left\{h(y):~y\in \cC \right\} \geq \min_y\left\{F(y):~y\in \cC \right\}=:\Opt_f.
\]
Moreover, $\Opt_h=\Opt_f$ if and only if there exists an optimal solution $y^*$ to the problem on the right-hand side satisfying $\alpha_j c_j(y^*)=0$ for some $j\in\{1,\ldots,m\}$.
\end{lemma}
\begin{proof}
First, we note that for any $y\in \cC$, we have $\alpha_j c_j(y) \geq 0$ since $\alpha_j \leq 0$, and thus for all $j\in\{1,\ldots,m\}$, $f_j(y) = h(y) - \alpha_j\, c_j(y) \leq h(y)$. This establishes $\Opt_h \geq \Opt_f$.

Let $y^*$ be an optimal solution to $\min_y\left\{F(y):~y\in \cC\right\}$ for which $\alpha_j c_j(y^*) = 0$ for some $j$. Then we have $F(y^*) = f_j(y^*)=h(y^*)$, which implies that $y^*$ is also optimal to $\Opt_h$. Now consider the case where every optimal solution $y^*\in\argmin_y\left\{F(y):~y\in \cC\right\}$ satisfies $\alpha_j c_j(y^*) > 0$ for all $j$. Note that for any $y \in \cC$ satisfying $\alpha_j c_j(y)>0$ for all $j$, we have $F(y)<h(y)$. Thus, for such optimal solutions $y^*$, we have $F(y^*) < h(y^*)$, and for any other non-optimal solution $y \in \cC$, we have $F(y^*) < F(y) \leq h(y)$, which implies $\Opt_f < \Opt_h$.
\end{proof}

Let us now turn our attention back to the TRS \eqref{eqn:trsSOC-conic}. Henceforth, we define $h(y) := y^\top Q y + 2 g^\top y$ to be our nonconvex quadratic objective function, where $Q$ is some symmetric matrix with $\lambda_Q < 0$. It is easy to see that on any bounded domain $\cC$, $h(y)$ has no local minimum on $\intt(\cC)$.
Hence, Observation~\ref{obs:boundary-nonconvex} points out the important role of the boundary of the domain $\left\{ y :~ \|y\| \leq 1, \ Ay - b \in \cK \right\}$ to the TRS \eqref{eqn:trsSOC-conic}.

A possible convex relaxation for \eqref{eqn:trsSOC-conic} suggested by Lemma~\ref{lem:nonconvex-underestimator-tightness} is that we embed the conic constraints $Ay - b \in \cK$ into the ground set $C$ and aggregate the constraint $\|y\| \leq 1$ with weight $\alpha=\lambda_Q$ to obtain the objective function
\begin{equation}\label{eqn:f-function}
f(y) := h(y) + \lambda_Q(1-\|y\|^2) = y^\top (Q - \lambda_Q I_n) y + 2g^\top y + \lambda_Q.
\end{equation}
Note $Q - \lambda_Q I_n \succeq 0$, and thus the function $f(y)$ is convex, and clearly is also an underestimator of $h(y)$, hence minimizing $f(y)$ over our domain is still a convex relaxation. Lemma~\ref{lem:nonconvex-underestimator-tightness} then gives us a precise characterization for when the convex relaxation using $f(y)$ is tight.
\begin{corollary}\label{cor:trsSOC-tightness}
Suppose $\lambda_Q<0$. Consider the convex relaxation for problem~\eqref{eqn:trsSOC-conic} given by
\begin{equation}\label{eqn:convex-trsSOC-conic}
\Opt_{f} = \min_y \left\{ f(y) : \begin{array}{rcl} \|y\| &\leq& 1\\ Ay - b &\in& \cK \end{array} \right\},
\end{equation}
where $f(y)$ is defined in \eqref{eqn:f-function}. This convex relaxation is tight if and only if there exists an optimal solution $y^*$ to \eqref{eqn:convex-trsSOC-conic} such that $\|y^*\| = 1$.
\end{corollary}

Because $Q - \lambda_Q I_n$ is not full rank, when $g$ is not orthogonal to $\Null(Q - \lambda_Q I_n)$, it is easy to see that the function $f(y)$ has no local minima on the interior of our domain. Then by Observation~\ref{obs:boundary-nonconvex}, the optimal solutions to \eqref{eqn:convex-trsSOC-conic} lie on $\bd(\left\{ y : \|y\| \leq 1, \ Ay - b \in \cK \right\})$. When $g$ is orthogonal to $\Null(Q - \lambda_Q I_n)$, then we can add $d \in \Null(Q - \lambda_Q I_n)$ to any point $y$ without changing the objective $f(y + d)$, hence there will always exist an optimal solution of \eqref{eqn:convex-trsSOC-conic} on $\bd(\left\{ y : \|y\| \leq 1, \ Ay - b \in \cK \right\})$. However, $f(y) = h(y)$ if and only if $\|y\| = 1$, but $f(y)$ may not be equal to $h(y)$ on all of $\bd(\left\{ y : \|y\| \leq 1, \ Ay - b \in \cK \right\})$. More precisely, we will have $f(y) < h(y)$ for $y \in \bd(\left\{ y : \|y\| \leq 1, \ Ay - b \in \cK \right\}) \cap \left\{ y : \|y\| < 1 \right\}$, so if all minima of $f(y)$ lie on this set, the convex relaxation \eqref{eqn:convex-trsSOC-conic} will not be tight.
Therefore, we next state a sufficient condition that ensures that there is always an optimal solution of \eqref{eqn:convex-trsSOC-conic} on the boundary of the unit ball.

\begin{condition}\label{cond:TRS-relaxation}
There exists a vector $d \neq 0$ such that $Qd = \lambda_Q d$, $Ad \in \cK$ and $g^\top d \leq 0$.
\end{condition}

\begin{theorem}\label{thm:trsSOC-tight-convex}
Suppose that $\lambda_Q < 0$ and that Condition~\ref{cond:TRS-relaxation} holds for the TRS given in \eqref{eqn:trsSOC-conic}. Then the convex relaxation given by \eqref{eqn:convex-trsSOC-conic} is tight.
\end{theorem}
\begin{proof}
Let $y^*$ be an optimum solution for \eqref{eqn:convex-trsSOC-conic}. 
If  $\|y^*\| = 1$, then from Corollary~\ref{cor:trsSOC-tightness}, the result follows immediately. Hence, we assume $\|y^*\| < 1$.

Let $d \neq 0$ be the vector from Condition~\ref{cond:TRS-relaxation}, thus $Qd = \lambda_Q d$, $Ad \in \cK$ and $g^\top d \leq 0$. Then for any $\epsilon > 0$, $A(y^* + \epsilon d) - b = (Ay^* - b) + \epsilon A d \in \cK$ because $\cK$ is a convex cone and $Ad \in \cK$ by assumption. Because $\|y^*\| < 1$, we may increase $\epsilon$ until $\|y^* + \epsilon d\| = 1$ and the vector $y^* + \epsilon d$ is still feasible. Note $(Q - \lambda_Q I_n) d = 0$, so for any $\epsilon > 0$,
\[
f(y^* + \epsilon d) = f(y^*) + 2(g^\top d) \epsilon \leq f(y^*).
\]
If $g^\top d < 0$, this violates optimality of $y^*$ since $\epsilon > 0$, thus $g^\top d = 0$. Then the vector $y^* + \epsilon d$ is an alternative optimum solution to \eqref{eqn:convex-trsSOC-conic} satisfying $\|y^* + \epsilon d\|=1$. Hence, the tightness of the relaxation \eqref{eqn:convex-trsSOC-conic} follows from Corollary~\ref{cor:trsSOC-tightness}.
\end{proof}

\begin{remark}\label{rem:cond:TRS-relaxation}
From the definition of $\lambda_Q$, Condition~\ref{cond:TRS-relaxation} is immediately satisfied for the classical TRS~\eqref{eqn:trsSOC-classical}  without additional conic constraints, i.e., when $A=I_n$, $b=0$, and $\cK=\R^n$.
\epr
\end{remark}

Consequently, in the case of  classical TRS, Remark~\ref{rem:cond:TRS-relaxation} implies the following specialization of Theorem~\ref{thm:trsSOC-tight-convex}. 
\begin{theorem}\label{thm:trsSOC-classical-tight-convex}
When $\lambda_Q<0$, a tight convex relaxation of classical TRS \eqref{eqn:trsSOC-classical} is given by 
\begin{equation}\label{eqn:convex-trsSOC-classical}
\Opt_{f} = \min_y \left\{ f(y) := y^\top (Q - \lambda_Q I_n) y + 2g^\top y + \lambda_Q :~ \|y\| \leq 1 \right\}.
\end{equation}
\end{theorem}

\begin{remark}\label{rem:FortinWolkowicz}
In order to handle a particular `hard case' of classical TRS,  
Fortin and Wolkowicz \cite{FortinWolkowicz2004} introduce and analyze the convex reformulation \eqref{eqn:convex-trsSOC-classical} (see \cite[Lemma 2.3]{FortinWolkowicz2004} and \cite[Section 7]{FortinWolkowicz2004}). 
We believe \eqref{eqn:convex-trsSOC-classical} can be of more use than stated in \cite{FortinWolkowicz2004}. In particular, by re-analyzing \eqref{eqn:trsSOC-classical}, we are able to both 
\begin{enumerate}[(i)]
\item improve on the previously best-known theoretical convergence rate guarantees for solving the classical TRS (see Remark~\ref{rem:TRScompare} in Section~\ref{sec:SOCComplexity}), and 
\item establish the tightness of the convex reformulation \eqref{eqn:convex-trsSOC-conic} 
for TRS with conic constraints under appropriate conditions (see Theorem~\ref{thm:trsSOC-tight-convex}) and also for TRS with hollow constraints covering 
interval-bounded 
TRS (see \cite{BenTalTeboulle1996,BM14,PongWolkowicz2014,SternWolkowicz1995,YeZhang2003}), under a condition well-studied in the literature (see Corollary~\ref{cor:TRS-convexify-hollows} and Theorem~\ref{thm:TRS-convexify-hollows}).
\end{enumerate}
\epr
\end{remark}

\subsection{Discussion of Condition~\ref{cond:TRS-relaxation} and Related Conditions from the Literature}\label{sec:conditions}

For TRS with conic constraints  \eqref{eqn:trsSOC-conic}, Condition~\ref{cond:TRS-relaxation} is related to and generalizes many other conditions examined in the literature. 

A result similar to Theorem~\ref{thm:trsSOC-tight-convex} was implicitly proven by Jeyakumar and Li \cite{JeyakumarLi2013} under a dimensionality condition for the case of linear and conic quadratic constraints. We state the linear version of their condition below; the conic quadratic one is very similar.
\begin{condition}\label{cond:JeyakumarLi}
Consider the case of nonnegative orthant, i.e., $\cK = \R^m_+$. Suppose that the system of linear inequalities, i.e., the constraint $Ay - b \in \cK$ satisfies the requirement that $\dim(\Null(Q-\lambda_Q I_n)) \geq n-\dim(\Null(A)) +1$.
\end{condition}

\begin{lemma}\label{lem:JeyakumarLi}
Condition~\ref{cond:TRS-relaxation} generalizes the dimensionality condition of Jeyakumar and Li \cite{JeyakumarLi2013}, i.e., Condition~\ref{cond:JeyakumarLi}, stated for linear and conic quadratic constraints.
\end{lemma}
\begin{proof} 
Suppose Condition~\ref{cond:JeyakumarLi} holds. Then 
\[\dim(\Null(A)) + \dim(\Null(Q-\lambda_Q I_n)) \geq n+1;\] 
thus, there must exist $d \neq 0$ which is in the intersection $\Null(A)\cap \Null(Q-\lambda_Q I_n)$. That is,  $Qd = \lambda_Q d$ and $Ad = 0 \in \R^m_+ = \cK$. If $g^\top d \leq 0$, then Condition~\ref{cond:TRS-relaxation} holds with the vector $d$. If $g^\top d > 0$, then Condition~\ref{cond:TRS-relaxation} holds with the vector $d' = -d$.
\end{proof}

Jeyakumar and Li~\cite{JeyakumarLi2013} demonstrates that Condition~\ref{cond:JeyakumarLi} is satisfied in a number of cases related to the robust least squares and robust SOC programming problems. As a consequence of Lemma~\ref{lem:JeyakumarLi}, our Condition~\ref{cond:TRS-relaxation} is satisfied in these cases as well. That said, Condition~\ref{cond:TRS-relaxation} is more general than Condition~\ref{cond:JeyakumarLi} as demonstrated by the following example.

\begin{example}\label{ex:condJLdiff}
For the problem data given by
\[
Q = \begin{bmatrix}1 & 0\\ 0 & -1\end{bmatrix}, \quad g = \begin{bmatrix}1\\0\end{bmatrix}, \quad A = \begin{bmatrix}1 & -1\\ -1 & -1\end{bmatrix}, \quad b = {1\over 2}\begin{bmatrix} 1\\ 1 \end{bmatrix}, \quad \cK = \R^2_+,
\]
Condition~\ref{cond:TRS-relaxation} is satisfied with $d = [0;-1]$, but Condition~\ref{cond:JeyakumarLi} is not. 
\epr
\end{example}

Ben-Tal and  den Hertog \cite{BenTalDenHertog2014} and Locatelli \cite{Locatelli2016} study a different SOC-based convex relaxation of \eqref{eqn:trsSOC-conic} given in a lifted space when $Q$ is a diagonal matrix and the additional constraints are linear, i.e., $\cK = \bbR^m_+$. Let $Q = \Diag(\{q_1,\ldots,q_n\})$; then this reformulation is given by 
\begin{equation}\label{eqn:lifted_relaxation}
\min_{y,z} \left\{ \sum_{i=1}^{n} q_i z_i + 2g^\top y:~\begin{array}{l} y_i^2 \leq z_i, \quad i=1,\ldots,n\\
\sum_{i=1}^{n} z_i \leq 1\\
A y \geq b
\end{array}\right\}.
\end{equation}
It was established in \cite{BenTalDenHertog2014} that for the classical TRS this convex reformulation is tight. 
Tightness of this relaxation for the TRS with additional linear constraints is studied in \cite{Locatelli2016} under the following condition:
\begin{condition}\label{cond:Locatelli}
Denote $Q = \Diag(\{q_1,\ldots,q_n\})$ and $J = \{j : q_j = \lambda_Q\}$. Also, define $A_J$ to be the matrix composed of columns of $A$ which correspond to the indices in $J$, and define $g_J$ analogously. For all $\epsilon > 0$, there exists $h_\epsilon$ with $\|h_\epsilon\| \leq \epsilon$ such that $\{\mu \geq 0 :~ A_J^\top \mu + g_J + h_\epsilon = 0\} = \emptyset$.
\end{condition}
\begin{lemma}\label{lem:Locatelli}
When $Q$ is diagonal and $\cK = \bbR^m_+$, Conditions~\ref{cond:TRS-relaxation}~and~\ref{cond:Locatelli} are equivalent.
\end{lemma}
\begin{proof}
It is shown in \cite[Proposition 3.3]{Locatelli2016} that Condition~\ref{cond:Locatelli} is equivalent to the program $\max_{\hat{y} \in \bbR^{|J|}} \{-g_J^\top \hat{y} : A_J \hat{y} \leq 0 \}$ being unbounded above or having multiple optima. In the former case, there must exist an extreme ray $\hat{d} \neq 0$ for which $g_J^\top \hat{d} < 0$ and $A_J \hat{d} \leq 0$. Setting $d$ to be the vector consisting of $\hat{d}$ in the $J$ entries and $0$ otherwise gives us $Q d = \lambda_Q d$, $Ad \leq 0$ and $g^\top d < 0$, which satisfies Condition~\ref{cond:TRS-relaxation}. In the latter case, we know that the zero vector is always a solution with objective value $0$, so having multiple optima means there exists $\hat{d} \neq 0$ such that $A_j \hat{d} \leq 0$ and $g_J^\top \hat{d} = 0$. Then a similar argument follows to show that Condition~\ref{cond:TRS-relaxation} holds.

Conversely, if Condition~\ref{cond:TRS-relaxation} holds, because $Q$ is diagonal, the vector $d$ given must have zeros everywhere except for entries in $J$. If $g^\top d < 0$, then the program above is unbounded, but if $g^\top d = 0$, then the program above has multiple optima since we can add $d$ to any optimal solution. Thus, Condition~\ref{cond:Locatelli} holds.
\end{proof}

\begin{remark}\label{rem:TestCond}
Condition~\ref{cond:TRS-relaxation} is equivalent to the conic program
\[ \min_d \left\{ g^\top d :~ (Q - \lambda_Q I_n) d = 0,\ Ad \in \cK \right\} \]
being unbounded below or having multiple optimal solutions. This follows from an extension of the proof of Lemma~\ref{lem:Locatelli} to the conic case. 
\epr
\end{remark}

\begin{remark}\label{rem:Locatelli}
Despite Condition~\ref{cond:TRS-relaxation} and  Condition~\ref{cond:Locatelli} being equivalent when $\cK=\R^m_+$, there are two major distinctions between our convex reformulation \eqref{eqn:convex-trsSOC-conic} and the one from \cite{BenTalDenHertog2014,Locatelli2016}. 
First, in order to diagonalize the matrix $Q$ in TRS and hence form the convex reformulation of \cite{BenTalDenHertog2014,Locatelli2016}, one needs to perform a full eigenvalue decomposition, 
which takes approximately $O(n^3)$ time and is more expensive than computing only the minimum eigenvalue (approximately $O(n^2)$ time) that is needed by our convex reformulation. 
Second, our convex reformulation \eqref{eqn:convex-trsSOC-conic} works in the original space of variables and thus preserves the nice structure of the domain, yet \eqref{eqn:lifted_relaxation} introduces new variables $z_1,\ldots,z_n$. Preserving the nice structure of the original convex domain becomes important when FOMs are applied to a convex reformulation of TRS. We discuss this issue in the case of classical TRS  in Section~\ref{sec:SOCComplexity}.
\epr
\end{remark}
  
\begin{remark}\label{rem:JeyakumarLi}
In contrast to the results given in \cite{JeyakumarLi2013} and \cite{Locatelli2016}, Theorem~\ref{thm:trsSOC-tight-convex} holds for general conic constraints when Condition~\ref{cond:TRS-relaxation} holds. Note that such general conic constraints can represent a variety of convex restrictions, and in particular, they may include positive semidefiniteness requirements. 
\epr
\end{remark}

We next present an example to illustrate that when Condition~\ref{cond:TRS-relaxation} is violated, we may not be able to give the exact convex reformulation. Moreover, a slight modification of this example demonstrates further that Condition~\ref{cond:TRS-relaxation} is not necessary for giving the exact convex reformulation. 
\begin{example}\label{ex:condNecessity1}
Suppose we are given the problem data:
\[
Q = \begin{bmatrix}1 & 0\\ 0 & -2\end{bmatrix}, \quad g = \begin{bmatrix}-3\\0\end{bmatrix}, \quad A = \begin{bmatrix}0 & 1\\ 0 & -1\end{bmatrix}, \quad b = {1\over 2}\begin{bmatrix} 1\\ 1 \end{bmatrix}, \quad \cK = \R^2_+.
\]
Then Condition~\ref{cond:TRS-relaxation} is violated. To see this, note that any $d$ satisfying $Qd = \lambda_Q d$ is of the form $d = [0; d_2]$. However, $Ad = [d_2; -d_2]$, so if $d_2 \neq 0$, $Ad \not\in \cK = \R^2_+$. For this problem data, $h(y) = y_1^2 - 2y_2^2 - 3y_1$ and $f(y) = 3y_1^2 - 3y_1 - 2$. It is easy to compute the minimizers of $f(y)$ over the unit ball to be the line $y_1 = 1/2$, with value $-11/4$. The constraints $Ay - b \in \cK$ are equivalent to $-1/2 \leq y_2 \leq 1/2$. 
\begin{figure}[htp!]
\centering
  \subfigure[$h(y) \leq -11/4 = \Opt_f$]{%
    \label{ex1-1}%
    \includegraphics[width=0.35\textwidth]{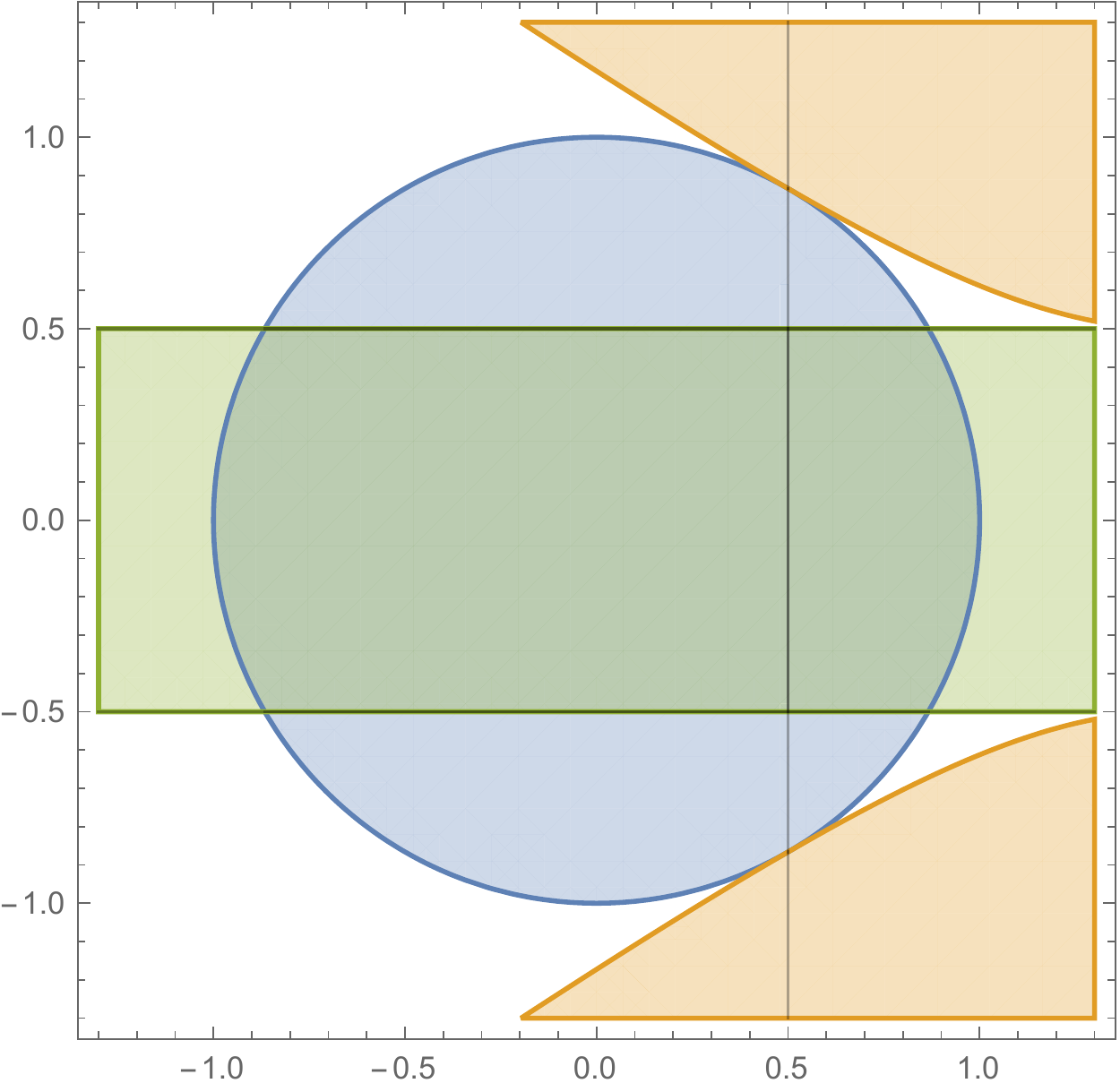}%
  } 
  \subfigure[$h(y) \leq (1-6\sqrt{3})/2 = \Opt_h$]{%
    \label{ex1-2}%
    \includegraphics[width=0.35\textwidth]{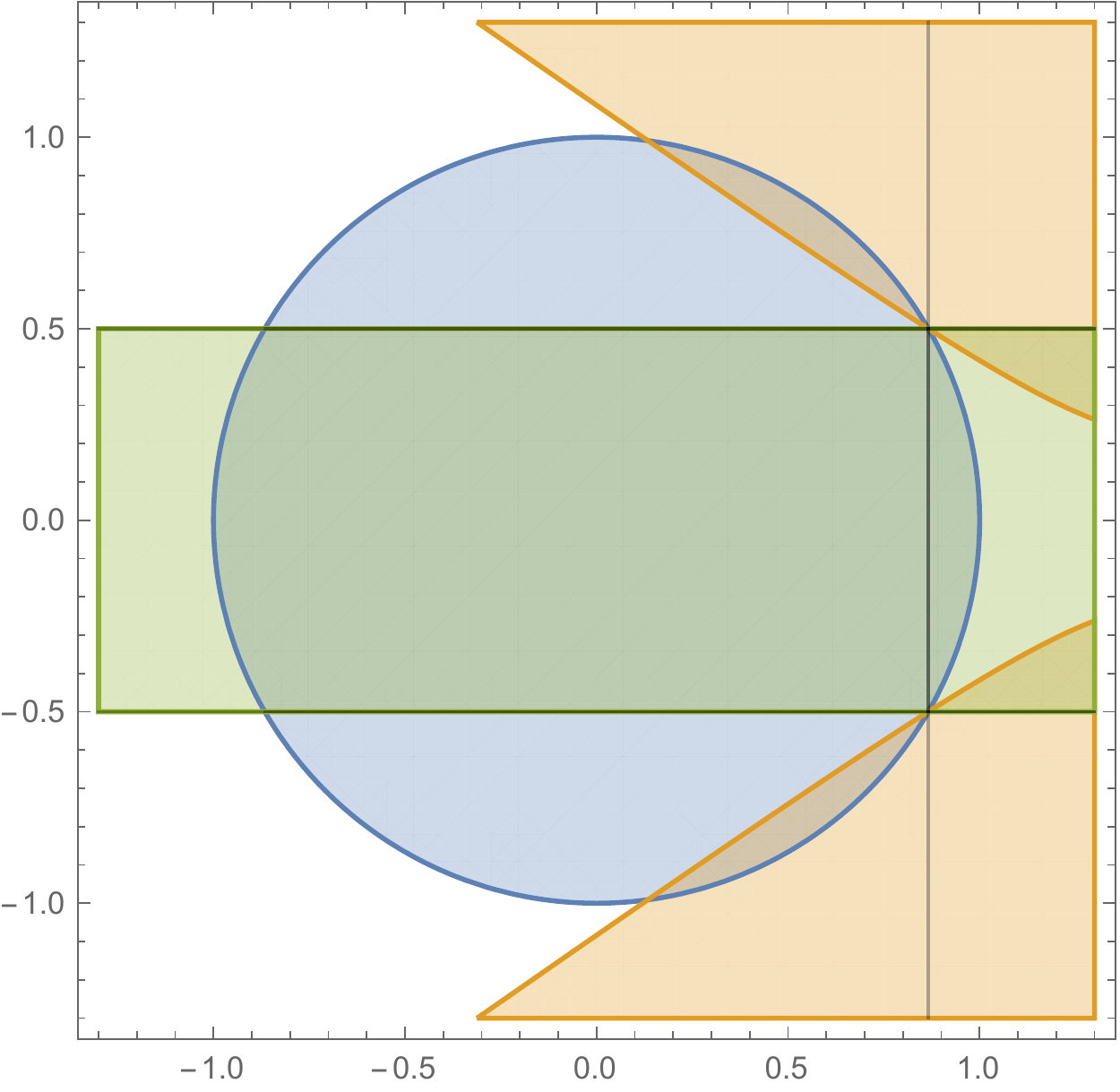}%
  } 
\caption{Contour plots of $h(y)$ over the feasible set.}
\label{fig:ex1}
\end{figure}

Figure~\ref{fig:ex1} shows that the minimizers of $h(y)$ over just the unit ball $\|y\| \leq  1$ lie on the boundary at $y = [1/2;\pm \sqrt{3}/2]$. Due to the linear constraints $-1/2 \leq y_2 \leq 1/2$, these points are cut off from the feasible region. As a result, any minimizer of $f(y)$ (i.e., the line $y_1 = 1/2$) inside the feasible region has norm strictly less than 1. Then by Corollary~\ref{cor:trsSOC-tightness}, the relaxation~\eqref{eqn:convex-trsSOC-conic} is not tight.

Finally, note that if we were to change our linear constraints to $-0.9 \leq y_2 \leq 0.9$, then our relaxation would be tight, while Condition~\ref{cond:TRS-relaxation} would still not be satisfied. However, for both cases in this example, the SDP relaxation of \cite{SturmZhang2003,YeZhang2003,BurerAnstreicher2013} strengthened with additional SOC-RLT inequalities is tight. 
\epr
\end{example}

A variant of Condition~\ref{cond:TRS-relaxation} is instrumental in giving exact convex hull characterization of the sets associated with the TRS  \eqref{eqn:trsSOC-conic}. We discuss these further in Section~\ref{sec:TRSConvexification}.

\subsection{Complexity of Solving Our Convex Reformulations}\label{sec:SOCComplexity}
In this section, we explore the complexity of solving our convex relaxation/reformulation of TRS via FOMs. 
Our convex relaxation/reformulation of TRS \eqref{eqn:convex-trsSOC-conic} and its variants have the same domain as their original nonconvex counterparts \eqref{eqn:trsSOC-conic} and thus are solvable via interior point methods and standard software as long as the cone $\cK$ has an explicit barrier function. However, because the standard polynomial-time  interior point methods have expensive iterations in terms of their dependence on the problem dimension, here we mainly focus on FOMs with cheap iterations. 
We next discuss the complexity of solving our convex reformulation of the classical TRS given by \eqref{eqn:trsSOC-classical} via Nesterov's accelerated gradient descent algorithm \cite{Nesterov_83}, an optimal FOM for this class of problems. 
Once again, the main distinction between solving  \eqref{eqn:convex-trsSOC-conic} as opposed to \eqref{eqn:convex-trsSOC-classical} via FOMs lies in how the projection onto the respective domain is handled. That is, whenever efficient projection onto the original domain is present, our discussion below will remain applicable to the conic case  \eqref{eqn:convex-trsSOC-conic} as well.

The reformulation \eqref{eqn:convex-trsSOC-classical} of classical TRS \eqref{eqn:trsSOC-classical} (or the convex relaxation \eqref{eqn:convex-trsSOC-conic} of TRS \eqref{eqn:trsSOC-conic}) is an SOC program (convex program) and can easily be built whenever $\lambda_Q$ is available to us. Moreover, computing $\lambda_Q$, the minimum eigenvalue of $Q$,  itself is a TRS with no linear term because 
\[
\lambda_Q = \min_y \left\{ y^\top Q y :~ \|y\| \leq 1 \right\}.
\]
There exist many efficient algorithms for computing the minimum eigenvalue of a symmetric matrix $Q$. One such algorithm that is effective for large sparse matrices is the Lanczos method \cite[Chapter 10]{GolubVanLoan1996book}. Implemented with a random start, this method enjoys the following probabilistic convergence guarantee (see \cite[Section 4]{KuczynskiWozniakowski1992estimating} and \cite[Section 5]{HazanKoren2016}): with probability at least $1-\delta$, the Lanczos method correctly estimates $\lambda_Q$ to within $\epsilon$-accuracy in $O\left( \sqrt{\|Q\|} \log(n/\delta) / \sqrt{\epsilon} \right)$ iterations. Furthermore, each iteration requires only matrix-vector products, and hence takes $O(N)$ time, where $N$ is the number of nonzero entries in $Q$. Consequently, with probability at least $1-\delta$, the randomized Lanczos method estimates $\lambda_Q$ to within $\epsilon$-accuracy in time $O\left( N \sqrt{\|Q\|} \log(n/\delta) / \sqrt{\epsilon} \right)$.

Given $\lambda_Q$, problem~\eqref{eqn:convex-trsSOC-classical}  
is simply minimizing a smooth convex quadratic function $f(y)$ with smoothness parameter $2(\lambda_{\max}(Q) - \lambda_Q) \leq 4\|Q\|$ over the unit ball. Therefore, this problem can be efficiently solved using Nesterov's accelerated gradient descent algorithm \cite{Nesterov_83},
which obtains an $\epsilon$-accurate solution in $O\left( \sqrt{\|Q\|} / \sqrt{\epsilon} \right)$ iterations. This is the optimal rate for FOMs for solving this class of problems. The major computational burden in each iteration in these FOMs is the evaluation of the gradient of $f(y)$, which involves simply a matrix-vector product, and hence each iteration costs $O(N)$ time. The only other main operation in each iteration of Nesterov's algorithm applied to this problem is the projection onto the Euclidean ball, and this can be done in $O(n)$ time. Consequently, Nesterov's algorithm \cite{Nesterov_83} applied to the optimization problem in our convex reformulation \eqref{eqn:convex-trsSOC-classical} of the classical TRS runs in time $O\left( N \sqrt{\|Q\|} / \sqrt{\epsilon} \right)$.

Thus, taking into account the complexity of computing $\lambda_Q$ to build our convex reformulation \eqref{eqn:convex-trsSOC-classical} and using Nesterov's algorithm \cite{Nesterov_83}, we establish the following upper bound on the worst case number of elementary operations needed:
\begin{theorem}\label{thm:classicalTRScomplexity}
With probability $1-\delta$, a solution $\bar{y}$ to the classical TRS \eqref{eqn:trsSOC-classical} satisfying $h(\bar{y}) - h(y) \leq \epsilon$ for all $y$ in the unit ball can be found in time
\begin{equation}\label{eqn:convex-trsSOC-classical-runtime}
O\left( N \left( \frac{ \sqrt{\|Q\|} }{\sqrt{\epsilon}} \log\left(\frac{n}{\delta}\right) + \frac{\sqrt{\|Q\|}}{\sqrt{\epsilon}} \right) \right) = O\left( N \frac{\sqrt{ \|Q\|} }{\sqrt{\epsilon}} \log\left(\frac{n}{\delta}\right) \right)
\end{equation}
using randomized Lanczos method to compute $\lambda_Q$ and Nesterov's algorithm \cite{Nesterov_83}. 
\end{theorem}

\begin{remark}\label{rem:TRScomplexity}
This discussion shows that the classical TRS decomposes into two special TRS problems: one without a linear term, i.e., $g = 0$, making it a pure minimum eigenvalue problem, and the other one with a convex quadratic objective function. This once again highlights the connection between the TRS and eigenvalue problems, and in fact demonstrates that, up to constant factors, the complexity of solving the classical TRS is no worse than solving a minimum eigenvalue problem because the complexity in Theorem~\ref{thm:classicalTRScomplexity} is essentially determined by the complexity of computing minimum eigenvalue of a matrix. 
Rendl and Wolkowicz \cite[Section 5]{Rendl_Wolkowicz_97} have  empirically observed this connection between complexity of solving classical TRS and computing the minimum eigenvalue; our analysis complements their study with a theoretical justification.
\epr
\end{remark}

\begin{remark}\label{rem:TRScomplexity2}
Using a deterministic algorithm to compute $\lambda_Q$ eliminates the probabilistic component in Theorem~\ref{thm:classicalTRScomplexity} at the expense of a slightly worse dependence on $\epsilon$ and $n$ in the iteration complexity. 

Unlike other methods \cite{FortinWolkowicz2004,More_Sorensen_83,Rendl_Wolkowicz_97}, our proposed method need not differentiate between the easy case and the hard case.
\epr
\end{remark}
 
\begin{remark}\label{rem:approx-eigenvalues}
In practice, we will not be able to form the objective $f(y)$ exactly, since  $\lambda_Q$ will be computed only approximately. Let us consider an estimate $\gamma \approx \lambda_Q$ and working with the objective $f_\gamma(y) = y^\top (Q - \gamma I_n) y + 2 g^\top y$. In Appendix~\ref{app:approx-eigenvalues}, we show that by using $f_\gamma(y)$ instead of $f(y)$, the error we incur is linearly dependent on the error of estimating $\lambda_Q$ with $\gamma$, which for our purposes is $O(\epsilon)$.
\epr
\end{remark}

\begin{remark}\label{rem:TRScompare}
Let us compare our bound \eqref{eqn:convex-trsSOC-classical-runtime} to the running time from \cite{HazanKoren2016}. The approach of \cite[Theorem 1]{HazanKoren2016} requires 
\[
O\left( N \frac{ \sqrt{\Gamma} \log\left( \Gamma / \epsilon \right)}{\sqrt{\epsilon}} \log\left(\frac{n}{\delta} \log\left(\frac{\Gamma}{\epsilon}\right)\right)  \right)
\]
elementary operations to obtain an $\epsilon$-accurate solution for \eqref{eqn:trsSOC-classical} with probability $1-\delta$, where $\Gamma = \max\left\{2( \|Q\| + \|g\|), 1\right\}$. 
By using the convex reformulation \eqref{eqn:convex-trsSOC-classical} as opposed to the method of \cite{HazanKoren2016}, we remove (at least) a factor of $\log\left( \Gamma / \epsilon\right)$ and the dependence on $\|g\|$. 

Our method is simpler to implement than the method of \cite{HazanKoren2016} as well because it decomposes the TRS into two well-studied problems as discussed in Remark~\ref{rem:TRScomplexity}. 
In contrast, since \cite{HazanKoren2016} relies on solving the dual SDP,  at the end of its iterations, it requires additional operations to obtain the primal solution from the dual one, and then a rounding procedure to find the solution in the original space. Also, because the approach of  \cite{HazanKoren2016} works in a lifted space and requires additional transformations at the end, it is not amenable to be used within the completely online convex optimization framework as described in \cite{Ho-NguyenKK2016RO}.  
\epr
\end{remark}

\section{Convexification of the Epigraph of TRS}
\label{sec:TRSConvexification}
In this section, we study the convex hull of the epigraph of TRS. 
In general, a tight convex relaxation for a nonconvex optimization problem does not necessarily imply that the epigraph of the convex relaxation is giving the exact convex hull of the epigraph of the nonconvex optimization problem. However, in the particular case of TRS with additional conic constraints, i.e., problem \eqref{eqn:trsSOC-conic}, under a slightly more stringent variant of Condition~\ref{cond:TRS-relaxation}, we will establish  that not only our convex relaxation given by \eqref{eqn:convex-trsSOC-conic} is tight but also its epigraph exactly characterizes the convex hull of the  epigraph of underlying TRS \eqref{eqn:trsSOC-conic} (see Corollary~\ref{cor:trs-convex-hull}).

By defining a new variable $x_{n+2}$ (where the variable $x_{n+1}$ is reserved for later homogenization), and moving the nonconvex function from the objective to the constraints, we can equivalently recast  \eqref{eqn:trsSOC-conic} as minimizing $x_{n+2}$ over its epigraph
\begin{equation}\label{eqn:trsSOC-conic2}
\Opt_h = \min_{y,x_{n+2}} \left\{ x_{n+2} : \begin{array}{rcl} \|y\| &\leq& 1\\ Ay - b &\in& \cK \\ h(y) = y^\top Q y + 2g^\top y &\leq& x_{n+2} \end{array} \right\}.
\end{equation}
Since the objective $x_{n+2}$ is linear, optimizing over the epigraph is equivalent to optimizing over its convex hull. 
We define the associated epigraph as 
\begin{equation}\label{eqn:X-set}
X := \left\{ x=[y;1;x_{n+2}] \in\R^{n+2}: \begin{array}{rcl} \|y\| &\leq& 1\\ Ay - b &\in& \cK \\ y^\top Q y + 2g^\top y &\leq& x_{n+2} \end{array} \right\}.
\end{equation}

Our convex hull characterizations are also SOC based. That is, as in Section~\ref{sec:TRSReformulation}, we focus mainly on the quadratic parts of the TRS \eqref{eqn:trsSOC-conic}, namely the nonconvex quadratic $y^\top Q y + 2g^\top y$ and the unit ball constraint $\|y\| \leq 1$ and provide the convexification of this set $X$ via a single new SOC constraint. 

Our approach is a refinement of the one from Burer and K{\i}l{\i}n\c{c}-Karzan~\cite{BKK14}. We first summarize the approach of \cite{BKK14} in Section~\ref{sec:BKKSummary} and then give our direct characterization in Section~\ref{sec:ourConvexHull}. As opposed to general SOCs and their cross-sections examined in Section~\ref{sec:BKKSummary},  we present a direct study of $\clconv(X)$ in Section~\ref{sec:ourConvexHull} that utilizes the fact that our domain in the context of TRS is a subset of an ellipsoid. Consequently, our analysis in Section~\ref{sec:ourConvexHull}  eliminates the need to verify several conditions from \cite{BKK14} completely and allows possibilities to handle additional conic constraints under appropriate assumptions. Finally, in Section~\ref{sec:TRS-additional-constraints}, we extend our analysis to cover additional hollow constraints in the domain.

\subsection{Summary and Discussion of Results from \cite{BKK14}}\label{sec:BKKSummary}
We start with a number of relevant definitions and conditions and then present the main result of \cite{BKK14}.

A cone $\cF^+\subseteq \R^k$ is said to be {\em second-order-cone representable} (or {\em SOCr}) if there exists a matrix $0\neq R \in \R^{k \times (k-1)}$ and a vector $r \in \R^k$ such that the nonzero columns of $R$ are linearly independent, $r \not\in \Range(R)$, and
\begin{equation} \label{equ:def:F+}
  \cF^+ = \left\{ x :~ \|R^\top  x \| \le r^\top  x \right\}.
\end{equation}

Given an SOCr cone $\cF^+$, the cone $\cF^-:=-\cF^+$ is also SOCr. Based on $\cF^+$ from \eqref{equ:def:F+}, we define $W := RR^\top  - rr^\top $ and consider the union $\cF^+ \cup (\cF^-)= \cF^+ \cup (-\cF^+)=: \cF$. Note that $\cF$ corresponds to a nonconvex cone defined by the homogeneous quadratic inequality $x^\top  W x \le 0$:
\[
 \cF :=  \cF^+ \cup (\cF^-) = \left\{ x :~ \|R^\top  x \|^2 \le (r^\top  x)^2 \right\} = \left\{ x :~ x^\top  W x \le 0 \right\}.
\]
We define $\apex(\cF^+)=\apex(\cF^-)=\apex(\cF)= \{x : R^\top x = 0,\ r^\top x = 0\}$. Any matrix $W$ of the form $W = R R^\top - r r^\top$ as described above has exactly one negative eigenvalue,  and given $\cF$, we can recover $\cF^+$ by performing an eigenvalue decomposition of $W$, see \cite[Propositions 1 and 3]{BKK14}.

Given matrices $W_0, W_1 \in \bbR^{k \times k}$ and a vector $h \in \bbR^k$, we let $W_t = (1-t) W_0 + t W_1$ for $t \in [0,1]$, and define  the sets
\begin{align*}
\cF_0 &:= \{x : x^\top W_0 x \leq 0\}, \quad \cF_1 := \{x : x^\top W_1 x \leq 0\}, \quad \cF_t = \{x : x^\top W_t x \leq 0\},\\
H^0 &:= \{x : h^\top x = 0\}, \quad H^1 := \{x : h^\top x = 1\}.
\end{align*}

Burer and K{\i}l{\i}n\c{c}-Karzan \cite{BKK14} provide a general scheme to build an SOC-based convex relaxation of $\cF_0^+ \cap \cF_1$ and establish that under appropriate conditions their relaxations are exactly  describing $\ccnh(\cF_0^+ \cap \cF_1)$ and $\clconv(\cF_0^+ \cap \cF_1 \cap H^1)$. Their analysis relies on the following conditions:
  
\begin{condition} \label{cond:one_neg_eval}
$W_0$ has at least one positive eigenvalue and exactly one negative eigenvalue.
\end{condition}

\begin{condition} \label{cond:interior}
There exists $\bar x$ such that $\bar x^\top  W_0 \bar x < 0$ and $\bar x^\top  W_1 \bar x < 0$. 
\end{condition}

\begin{condition} \label{cond:A0_apex}
Either (i) $W_0$ is nonsingular, (ii) $W_0$ is singular and $W_1$ is positive definite on $\Null(W_0)$, or (iii) $W_0$ is singular and $W_1$ is negative definite on $\Null(W_0)$.
\end{condition}

Conditions~\ref{cond:one_neg_eval}--\ref{cond:A0_apex} ensure the existence of a maximal $s \in [0,1]$ such that $W_t$ has a single negative eigenvalue for all $t \in [0,s]$, $W_t$ is invertible for all $t \in(0,s)$, and $W_s$ is singular---that is, $\Null(W_s)$ is nontrivial whenever $s<1$.
Then, for all $W_t$ with $t \in [0,s]$, the set $\cF_t^+$ is well-defined by computing an eigenvalue decomposition of $W_t$. We also need the following conditions on the value of $s$:

\begin{condition} \label{cond:As_null}
When $s < 1$, $\apex(\cF_s^+) \cap \intt(\cF_1) \ne \emptyset$. 
\end{condition}

\begin{condition} \label{cond:hyperplane}
When $s < 1$, $\apex(\cF_s^+)\cap \intt(\cF_1) \cap H^0 \ne \emptyset$
or $\cF_0^+ \cap \cF_s^+ \cap H^0 \subseteq \cF_1$.
\end{condition}

Conditions~\ref{cond:one_neg_eval}--\ref{cond:hyperplane} are all that is needed to state the main result of \cite{BKK14}. Here, we state \cite[Theorem 1]{BKK14} for completeness.
\begin{theorem}[{\cite[Theorem 1]{BKK14}}] \label{thm:BKK-convexify}
Suppose Conditions~\ref{cond:one_neg_eval}--\ref{cond:A0_apex} are satisfied, and let $s$ be the maximal $s \in [0,1]$ such that $W_t:=(1-t)W_0+t W_1$ has a single negative eigenvalue for all $t \in [0,s]$. Then $\ccnh(\cF_0^+ \cap \cF_1) \subseteq \cF_0^+ \cap \cF_s^+$, and equality holds under Condition~\ref{cond:As_null}. Moreover,  Conditions~\ref{cond:one_neg_eval}--\ref{cond:hyperplane} imply
$\cF_0^+ \cap \cF_s^+ \cap H^1 = \clconv(\cF_0^+ \cap \cF_1 \cap H^1).$
\end{theorem}

These convexification results were also applied to the classical TRS \eqref{eqn:trsSOC-classical} in \cite{BKK14}. In particular, it is shown in \cite[Section 7.2]{BKK14} that the classical TRS \eqref{eqn:trsSOC-classical} can be reformulated in the form of 
\begin{equation}\label{eqn:TRS-SOC-BKK}
\Opt_h = \min_{\tilde{y},x_{n+2}}\left\{ -x_{n+2}^2 : \begin{array}{rcl} \|\tilde{y}\| &\leq& 1\\ \tilde{y}^\top \tilde{Q} \tilde{y} + 2\tilde{g}^\top \tilde{y} &\leq& -x_{n+2}^2 \end{array} \right\},
\end{equation}
where $\tilde{g}=[g;0]$ and $\tilde{Q}:=\begin{bmatrix} Q & 0\\ 0 &\lambda_Q\end{bmatrix}$ is defined to ensure $\lambda_{\min}(\tilde{Q})=\lambda_Q$ and the multiplicity of $\lambda_Q$ in $\tilde{Q}$ is at least two. Note that here $\tilde{y} = [y;\tilde{y}_{n+1}]\in\R^{n+1}$.
Then \cite{BKK14} suggests to solve \eqref{eqn:TRS-SOC-BKK} in two stages after the nonconvex domain in \eqref{eqn:TRS-SOC-BKK} is replaced by its convex hull. 
Specifically, \cite{BKK14} defines a new variable $\tilde{x}=[\tilde{y}; x_{n+1}; x_{n+2}]$ and the matrices
\begin{equation}\label{eqn:BKK-matrices}
\tilde{W}_0 = \begin{bmatrix} I_{n+1} & 0 & 0\\ 0^\top & -1 & 0\\ 0 & 0 & 0 \end{bmatrix}, \quad \tilde{W}_1 = \begin{bmatrix} \tilde{Q} & \tilde{g} & 0\\ \tilde{g}^\top & 0 & 0\\ 0 & 0 & 1 \end{bmatrix}, 
\end{equation}
which then leads to 
\begin{align}\label{eqn:BKK-nonconvex-set}
Y&:=\left\{ [\tilde{y};1;x_{n+2}]\in\R^{n+3} :~ \begin{array}{rcl} \|\tilde{y}\| \!&\!\leq\!& 1\\ \tilde{y}^\top \tilde{Q} \tilde{y} + 2\tilde{g}^\top \tilde{y} \!&\!\leq\!& -x_{n+2}^2 \end{array} \right\} \notag\\
&= \left\{ \tilde{x} = [\tilde{y};x_{n+1};x_{n+2}] \in\R^{n+3} :~ \begin{array}{rcl} \tilde{x}^\top \tilde{W}_0 \tilde{x} \!&\!\leq\!&\! 0\\ \tilde{x}^\top \tilde{W}_1 \tilde{x} \!&\!\leq\!&\! 0\\ x_{n+1} \!&\!=\!&\! 1 \end{array} \right\}\notag\\
&= \cF_0^+ \cap \cF_1 \cap \left\{ \tilde{x}\in\R^{n+3} :~ x_{n+1} = 1 \right\},
\end{align}
where $\cF_0 = \{\tilde{x} : \tilde{x}^\top \tilde{W}_0 \tilde{x} \leq 0\}$ and $\cF_1 = \{\tilde{x} : \tilde{x}^\top \tilde{W}_1 \tilde{x} \leq 0\}$. 
Then the conditions of Theorem~\ref{thm:BKK-convexify} are satisfied, and we deduce that there exists some $s \in (0,1)$ ensuring
\begin{equation}\label{eqn:BKK-TRS-hull}
\clconv(\cF_0^+ \cap \cF_1 \cap \left\{ \tilde{x} : x_{n+1} = 1 \right\}) = \cF_0^+ \cap \cF_s^+ \cap \left\{ \tilde{x} : x_{n+1} = 1 \right\}.
\end{equation}
While the precise value of $s$ is not given in \cite{BKK14}, one can show that in fact $s = \frac{1}{1-\lambda_Q}$. 
We present the verification of conditions of Theorem~\ref{thm:BKK-convexify} for matrices in \eqref{eqn:BKK-matrices} and the derivation for this $s$ value in Appendix~\ref{app:s-computation}.

\begin{remark}\label{rem:BKK-nonpositive}
The reformulation \eqref{eqn:TRS-SOC-BKK} of classical TRS \eqref{eqn:trsSOC-classical} implicitly requires that $\Opt_h \leq 0$ because of the constraint $\tilde{y}^\top \tilde{Q} \tilde{y} + 2\tilde{g}^\top \tilde{y} \leq -x_{n+2}^2 \leq 0$. For the classical TRS \eqref{eqn:trsSOC-classical} with no  additional constraints, this is not an additional limitation because $\tilde{y}=0$ will always be a feasible solution with objective value $0$ and thus the optimum solution will have a nonpositive objective value. However, this becomes a limitation when we want to extend such arguments for the TRS \eqref{eqn:trsSOC-conic} with additional conic constraints $Ay - b \in \cK$ because $\Opt_h$ may no longer be nonpositive.
\epr
\end{remark}

\subsection{Direct Convexification of the Epigraph of TRS}\label{sec:ourConvexHull}

Due to Remark~\ref{rem:BKK-nonpositive}, we instead choose to study the epigraph of TRS \eqref{eqn:trsSOC-conic} as in \eqref{eqn:X-set}, which allows for positive objective values in \eqref{eqn:trsSOC-conic2} and avoids the additional lifting of the problem $Q \to \tilde{Q}$. 
To this end, we define the matrices
\begin{equation}\label{eqn:trs-matrices-defn}
W_0 = \begin{bmatrix} I_n & 0 & 0\\ 0^\top & -1 & 0\\ 0 & 0 & 0 \end{bmatrix}, \quad W_1 = \begin{bmatrix} Q & g & 0\\ g^\top & 0 & -\frac{1}{2}\\ 0 & -\frac{1}{2} & 0 \end{bmatrix},
\end{equation}
and the corresponding sets
\begin{align}\label{eqn:trs-cones-defn}
\cF_0^+ &= \left\{ x=[y;x_{n+1};x_{n+2}]\in\R^{n+2} :~ \|y\|^2 \leq x_{n+1}^2, \ x_{n+1} \geq 0 \right\} \notag\\
&= \left\{ x\in\R^{n+2} :~ x^\top W_0 x \leq 0, \ x_{n+1} \geq 0 \right\}, \notag\\
\cF_1 &= \left\{ x\in\R^{n+2} :~ y^\top Q y + 2g^\top y\, x_{n+1} \leq x_{n+1} x_{n+2} \right\} = \left\{ x :~ x^\top W_1 x \leq 0\right\},\\
\widehat{\cK} &= \left\{ x\in\R^{n+2} : Ay - b x_{n+1} \in \cK \right\},\notag\\
H^1 &= \left\{ x\in\R^{n+2} :~ x_{n+1} =1 \right\}.\notag
\end{align}
Note that $\lambda_Q<0$, and thus $\cF_1$ is not convex.  
With these definitions, the epigraph $X$ from \eqref{eqn:X-set} can be written as
\[
X = \cF_0^+ \cap \cF_1 \cap \widehat{\cK} \cap H^1.
\]

It is mentioned in \cite{BKK14} that the matrices \eqref{eqn:trs-matrices-defn} do not satisfy the necessary conditions to apply Theorem~\ref{thm:BKK-convexify} directly. In particular, Condition~\ref{cond:A0_apex} is violated for the choice of matrices \eqref{eqn:trs-matrices-defn}. As a result, \cite[Section 7.2]{BKK14} reformulates the  classical TRS with matrices \eqref{eqn:BKK-matrices} instead.
In contrast, we next show that in the special case of the classical TRS,  via a direct analysis, finding the convex hull through linear aggregation of constraints will still carry through for the matrices in \eqref{eqn:trs-matrices-defn}. This then indicates that while Condition~\ref{cond:A0_apex} is sufficient, it is not necessary to obtain the convex hull result. In fact, we show that the value of $s = \frac{1}{1-\lambda_Q}$ that works for the matrices \eqref{eqn:BKK-matrices} will also work for our matrices \eqref{eqn:trs-matrices-defn}.  
More precisely, for $s = \frac{1}{1-\lambda_Q}$, we define
\begin{equation}\label{eqn:trs-F_s-defn}
\cF_s = \left\{ x : x^\top W_s x \leq 0 \right\} = \left\{ x : y^\top (Q - \lambda_Q I_n) y + 2g^\top y x_{n+1} + \lambda_Q x_{n+1}^2 \leq x_{n+1} x_{n+2} \right\},
\end{equation}
and prove that $\clconv(X)=\conv(X)=\conv(\cF_0^+ \cap \cF_1 \cap \widehat{\cK}\cap H^1) = \cF_0^+ \cap \cF_s \cap \widehat{\cK}\cap H^1$ directly under the following condition that handles additional conic constraints.
\begin{condition}\label{cond:TRS-convexify}
There exists a vector $d \neq 0$ such that $Qd = \lambda_Q d$,  $ Ad \in \cK$, and $ -Ad \in \cK$.
\end{condition}
Note that when $\cK$ is pointed and $A$ is full rank, Condition~\ref{cond:TRS-convexify} assumes $Ad=0$.

\begin{remark}\label{rem:cond:TRS-convexify}
Condition~\ref{cond:TRS-convexify} implies  Condition~\ref{cond:TRS-relaxation}. To see this, suppose $d \neq 0$ satisfies Condition~\ref{cond:TRS-convexify}. Then if $g^\top d \leq 0$, $d$ satisfies Condition~\ref{cond:TRS-relaxation} also. Otherwise, $-d$ will satisfy Condition~\ref{cond:TRS-relaxation}. We demonstrate that Condition~\ref{cond:TRS-relaxation} does not imply Condition~\ref{cond:TRS-convexify}  in Example~\ref{ex:condNecessityConvexHull}. 

Furthermore, Condition~\ref{cond:TRS-convexify} holds whenever Condition~\ref{cond:JeyakumarLi} of \cite{JeyakumarLi2013} is satisfied  because Condition~\ref{cond:JeyakumarLi} implies that there exists $d$ such that $Qd = \lambda_Q d$ and $Ad = 0$ and since $\cK$ is a closed convex cone, $\pm Ad = 0 \in \cK$ as well. 
\epr
\end{remark}

One of the ingredients of our convex hull result is given in the next lemma.
\begin{lemma}\label{lem:trs-conic-convexify1}
Let $\cF_s$ be defined as in \eqref{eqn:trs-F_s-defn}. Then the cone $\cF_s \cap \left\{ x : x_{n+1} > 0 \right\}$ is convex, and the set $\cF_s \cap H^1$ where $H^1$ is as defined in \eqref{eqn:trs-cones-defn} is SOC representable.
\end{lemma}
\begin{proof}
Let $x = [y;x_{n+1};x_{n+2}]\in \R^{n+2}$. Note that by definition, we have
\begin{align*}
&\cF_s \cap \{ x : x_{n+1} > 0 \} \\ 
& =  \left\{ x:~ y^\top (Q - \lambda_Q I_n) y + 2 g^\top y x_{n+1} + \lambda_Q x_{n+1}^2 \leq x_{n+1} x_{n+2}, ~x_{n+1} > 0  \right\} \\
& =  \left\{ x:~ y^\top (Q - \lambda_Q I_n) y \leq x_{n+1} (x_{n+2} - 2 g^\top y - \lambda_Q x_{n+1}), ~x_{n+1} > 0 \right\} \\
& =  \left\{ x:~ \begin{array}{l} y^\top (Q - \lambda_Q I_n) y \leq x_{n+1} (x_{n+2} - 2 g^\top y - \lambda_Q x_{n+1}),\\ x_{n+1} > 0, ~~x_{n+2} - 2 g^\top y - \lambda_Q x_{n+1} \geq 0 \end{array}\right\}, 
\end{align*}
where the last equation follows because $Q - \lambda_Q I_n \succeq 0$, we have $y^\top (Q - \lambda_Q I_n) y \geq 0$ for all $y$ and then $x_{n+1} > 0$ implies $x_{n+2} - 2 g^\top y - \lambda_Q x_{n+1} \geq 0$. As a result, $x_{n+1} + x_{n+2} - 2 g^\top y - \lambda_Q x_{n+1}\geq 0$ holds for all $x\in \cF_s\cap  \{ x : x_{n+1} > 0 \}$. 
In addition, from these derivations, we immediately deduce that the set $\cF_s\cap \{ x : x_{n+1} =1 \}$ is an SOC representable set. 
\end{proof}

\begin{theorem}\label{thm:trs-convexify}
Let $\cF_0^+, \cF_1,H^1,\widehat{\cK},\cF_s$ be defined as in \eqref{eqn:trs-cones-defn} and \eqref{eqn:trs-F_s-defn}. Assume that $\lambda_Q < 0$ and Condition~\ref{cond:TRS-convexify} holds. Then
\[
\clconv(\cF_0^+ \cap \cF_1 \cap \widehat{\cK} \cap H^1) = \cF_0^+ \cap \cF_s \cap \widehat{\cK} \cap H^1.
\]
\end{theorem}
\begin{proof}
We will first establish that $\conv(\cF_0^+ \cap \cF_1 \cap \widehat{\cK} \cap H^1) = \cF_0^+ \cap \cF_s \cap \widehat{\cK} \cap H^1$. Since the sets $\cF^+, \cF_s, \widehat{\cK}, H^1$ are closed, this will immediately imply our closed convex hull result.

It is clear from the definition of $\cF_s$ and Lemma~\ref{lem:trs-conic-convexify1} that $\conv(\cF_0^+ \cap \cF_1 \cap \widehat{\cK} \cap H^1) \subseteq \cF_0^+ \cap \cF_s \cap \widehat{\cK} \cap H^1$. We will prove the reverse direction. 

Let $x = [y;x_{n+1};x_{n+2}]$ be a vector in $\cF_0^+  \cap \widehat{\cK}\cap H^1 \cap \cF_s$. Then $x$ satisfies
\begin{align*}
x^\top W_0 x &\leq 0,\\
Ay - bx_{n+1} &\in \cK, \\
x_{n+1} &= 1,\\
x^\top W_s x &\leq 0.
\end{align*}
We will show that $x\in\conv(\cF_0^+ \cap \cF_1 \cap \widehat{\cK}\cap H^1)$. 
If $x \in \cF_1$, then we are done. Suppose $x \not\in \cF_1$, that is, $x^\top W_1 x > 0$. 
Then from the definition of $s$,  $0<x^\top W_1 x$, and $x^\top W_s x\leq 0$, we have 
\[
0<s(x^\top W_1 x) - x^\top W_s x=-(1-s) x^\top W_0 x =  {\lambda_Q \over 1-\lambda_Q} (x_{n+1}^2 - \|y\|^2).
\] 
Because $\lambda_Q< 0$, this implies $\|y\|^2 < x_{n+1}^2$. 
Let $d$ be the vector given by Condition~\ref{cond:TRS-convexify} such that $Qd = \lambda_Q d$, $Ad \in \cK$, $-Ad \in \cK$, and $\|d\|^2=1$. We now consider the points $x^\eta:= [y+\eta d;\, x_{n+1};\, x_{n+2} + 2g^\top d \eta]$ for $\eta\in\R$. 
We first argue that $x^\eta\in\cF_s$ holds for all $\eta\in\R$.
To see this, note that 
\begin{align} \label{eqn:proof:thm:convexify}
&(y+\eta d)^\top (Q - \lambda_Q I_n) (y + \eta d) + 2g^\top (y + \eta d) x_{n+1} + \lambda_Q x_{n+1}^2\notag \\
&= (y+\eta d)^\top Q (y + \eta d) + 2g^\top (y + \eta d) x_{n+1} + \lambda_Q (x_{n+1}^2 - \|y+\eta d\|^2) \notag\\
&= y^\top Q y + 2 \underbrace{y^\top Qd}_{=\lambda_Q y^\top d} \eta + \underbrace{d^\top Q d}_{=\lambda_Q} \eta^2 + 2g^\top y + 2 g^\top d x_{n+1} \eta \notag\\
&\quad + \lambda_Q (x_{n+1}^2 - \|y\|^2 - 2 y^\top d \eta - \eta^2) \notag\\
&= y^\top Q y + 2g^\top y x_{n+1} + \lambda_Q (x_{n+1}^2 - \|y\|^2) + 2g^\top d x_{n+1} \eta \notag\\
&= y^\top (Q - \lambda_Q I_n) y + 2g^\top y x_{n+1} + \lambda_Q x_{n+1}^2 + 2g^\top d x_{n+1} \notag \eta\\
&= x_{n+1} x_{n+2} + (1-\lambda_Q)(x^\top W_s x) + 2g^\top d x_{n+1} \eta \notag\\
&\leq x_{n+1} x_{n+2} + 2g^\top d x_{n+1} \eta\\
&= x_{n+1} (x_{n+2} + 2g^\top d \eta),\notag
\end{align} 
where the third equation follows from $Qd = \lambda_Q d$ and $\|d\|^2 = 1$, and the inequality holds because $x^\top W_s x \leq 0$ and $\lambda_Q<0$. 
Then from the inequality~\eqref{eqn:proof:thm:convexify} and the definition of $\cF_s$ in \eqref{eqn:trs-F_s-defn}, we conclude $x^\eta\in\cF_s$ for all $\eta\in\R$. 
Moreover, because $\|y\|^2 < x_{n+1}^2$ and $d \neq 0$, there must exist $\delta, \epsilon > 0$ such that $\|y-\delta d\|^2 = \|y + \epsilon d\|^2 = x_{n+1}^2$. 
We define
\begin{align*}
x^\delta &:= [y-\delta d;\, x_{n+1};\, x_{n+2} - 2g^\top d \delta]\\
x^\epsilon &:= [y+\epsilon d;\, x_{n+1};\, x_{n+2} + 2g^\top d \epsilon].
\end{align*}
Then by our choice of $\delta, \epsilon$, we have $x^\delta, x^\epsilon \in \bd(\cF_0^+)$. 
From $s\in(0,1)$, $x^\eta\in\cF_s$ for all $\eta\in\R$, and the relation 
\[
(x^\eta)^\top W_s x^\eta= (1-s)[(x^\eta)^\top W_1 x^\eta] + s [(x^\eta)^\top W_0 x^\eta],
\]
we conclude that $x^\eta\in\cF_1$ for all $\eta$ such that $x^\eta\in \bd(\cF_0^+)$. In particular,  $x^\delta, x^\epsilon \in \cF_1$. 
Furthermore, by Condition~\ref{cond:TRS-convexify}, $\pm Ad \in \cK$, and since $\cK$ is a cone, $- Ad \delta, Ad \epsilon \in \cK$; thus $x^\delta, x^\epsilon \in \widehat{\cK}$. Finally, $x_{n+1} = 1$ in both $x^\delta, x^\epsilon$, and so we have $x^\delta, x^\epsilon \in \cF_0^+ \cap \cF_1 \cap \widehat{\cK} \cap H^1$. Now it is easy to see that
\[ x = \frac{\epsilon}{\delta + \epsilon} x^\delta + \frac{\delta}{\delta + \epsilon} x^\epsilon \in \conv(\cF_0^+ \cap \cF_1 \cap \widehat{\cK} \cap H^1).\]

As a consequence, we have the relation
\[ \cF_0^+ \cap \cF_1 \cap \widehat{\cK} \cap H^1 \subseteq \cF_0^+ \cap \cF_s \cap \widehat{\cK} \cap H^1 \subseteq \conv(\cF_0^+ \cap \cF_1 \cap \widehat{\cK} \cap H^1). \]
By Lemma~\ref{lem:trs-conic-convexify1}, the set $\cF_s \cap H^1$ is SOC representable and hence convex; this implies that $\cF_0^+ \cap \cF_s \cap \widehat{\cK} \cap H^1$ is convex also. Then taking the convex hull of all terms in the above inequality gives us the result.
\end{proof}

Note that the set $\cF_0^+ \cap \cF_s \cap \widehat{\cK}\cap H^1$ is closed. We give our explicit convex hull result for TRS below.
\begin{corollary}\label{cor:trs-convex-hull}
Let $X$ be the set defined in \eqref{eqn:X-set}. When $\lambda_Q < 0$, under Condition~\ref{cond:TRS-convexify} we have
\[
\conv(X) = \left\{ x = [y;1;x_{n+2}] : \begin{array}{rcl} \|y\| &\leq& 1\\ y^\top (Q - \lambda_Q I_n) y + 2g^\top y + \lambda_Q &\leq& x_{n+2}\\ Ay - b &\in& \cK \end{array} \right\}.
\]
As a result,
\begin{align*}
\Opt_h &= \min_y \left\{ h(y) = y^\top Q y + 2 g^\top y : \begin{array}{rcl} \|y\| &\leq& 1\\ Ay - b &\in& \cK \end{array} \right\}\\
&= \min_y \left\{ f(y) = y^\top (Q-\lambda_Q I_n) y + 2 g^\top y + \lambda_Q : \begin{array}{rcl} \|y\| &\leq& 1\\ Ay - b &\in& \cK \end{array} \right\}.
\end{align*}
\end{corollary}

\begin{remark}\label{rem:TRSExactConvexHull}
In the particular case of TRS with additional conic constraints, i.e., problem \eqref{eqn:trsSOC-conic}, under Condition~\ref{cond:TRS-convexify}, Corollary~\ref{cor:trs-convex-hull} shows that not only our convex relaxation given by \eqref{eqn:convex-trsSOC-conic} is tight but also we can characterize the convex hull of its epigraph exactly. Because Condition~\ref{cond:TRS-convexify} holds for the classical TRS, this then recovers the results from \cite[Section 6.2]{BKK14}.
\epr
\end{remark}

As a consequence of Remark~\ref{rem:cond:TRS-convexify} and  Corollary~\ref{cor:trs-convex-hull}, in all of the cases where Jeyakumar and Li \cite{JeyakumarLi2013} show the tightness of their convex reformulation, i.e., for robust least squares and robust SOC programming, we can further give the exact convex hull characterizations of the associated epigraphs.

We next present an example to illustrate that when Condition~\ref{cond:TRS-convexify} is violated, we may not be able to obtain the convex hull description. We also give a variant of this example to demonstrate that there are cases where our convex relaxation is tight while Condition~\ref{cond:TRS-convexify} is still violated.
\begin{example}\label{ex:condNecessityConvexHull}
Consider the following problem with the data given by
\[
Q = \begin{bmatrix}1 & 0\\ 0 & -1\end{bmatrix}, \quad g = \begin{bmatrix}0 \\ 1\end{bmatrix}, \quad A = \begin{bmatrix}0 & -1\end{bmatrix}, \quad b = \frac{1}{2}, \quad \cK = \R_+.
\]
In this example, Condition~\ref{cond:TRS-convexify} is violated. To see this, any vector $d$ such that $Qd = \lambda_Q d$ is of the form $d = [0;d_2]$. But then $A d = -d_2$. Hence, if $d_2 > 0$ then $Ad \not\in \cK$, and similarly, if $d_2 < 0$ then $-Ad \not\in \cK$.
\begin{figure}[htp]
\centering
  \subfigure[$X = \cF_0^+ \cap \cF_1 \cap \widehat{\cK} \cap H^1$]{%
    \label{ex2F0F1}%
    \includegraphics[width=0.32\textwidth]{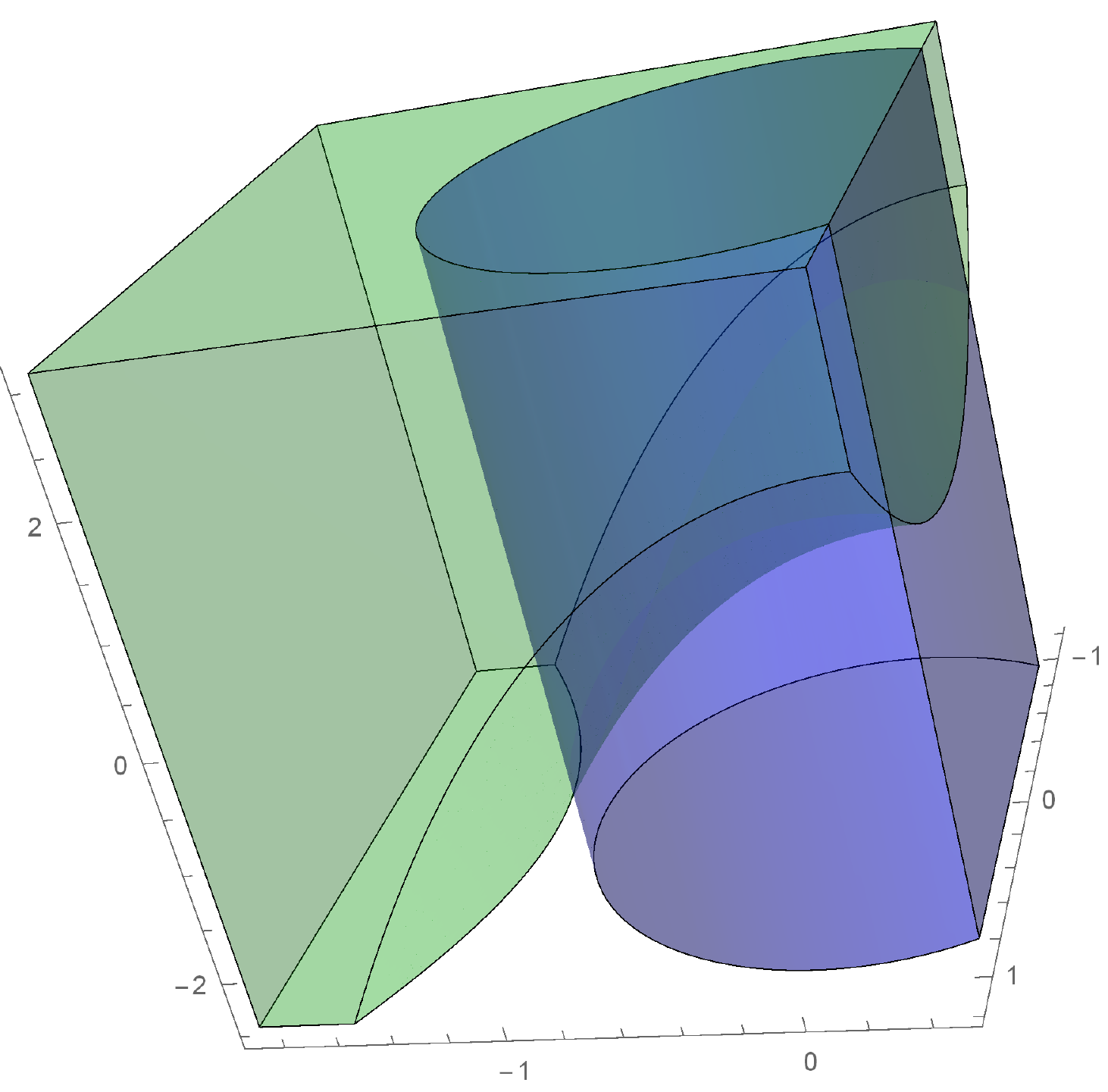}%
  } 
  \subfigure[$\cF_0^+ \cap \cF_s \cap \widehat{\cK} \cap H^1$]{%
    \label{ex2F0Fs}%
    \includegraphics[width=0.32\textwidth]{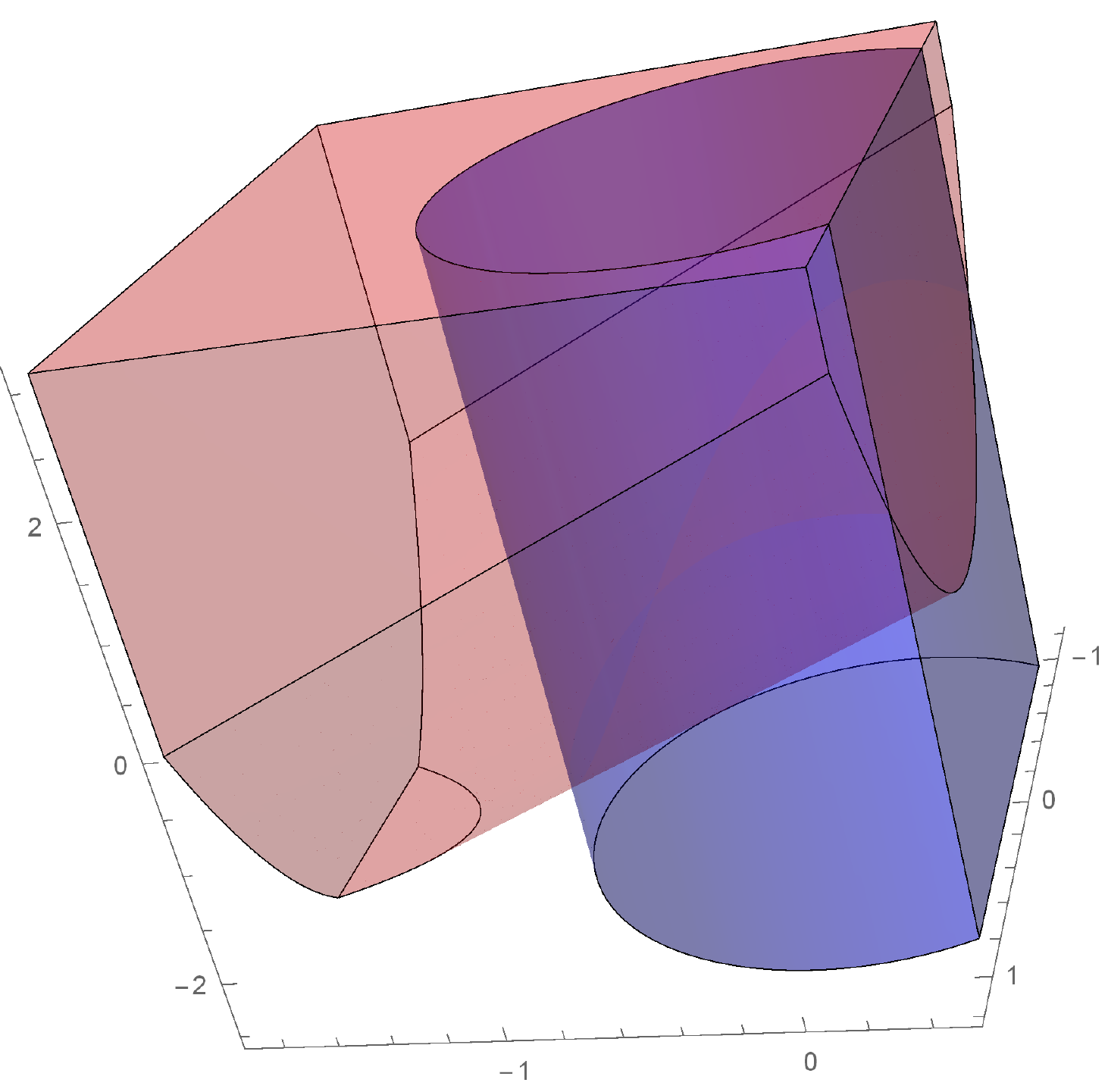}%
  }
  \subfigure[$X$ vs $\cF_0^+ \cap \cF_s \cap \widehat{\cK} \cap H^1$]{%
    \label{ex2F0F1Fs}%
    \includegraphics[width=0.32\textwidth]{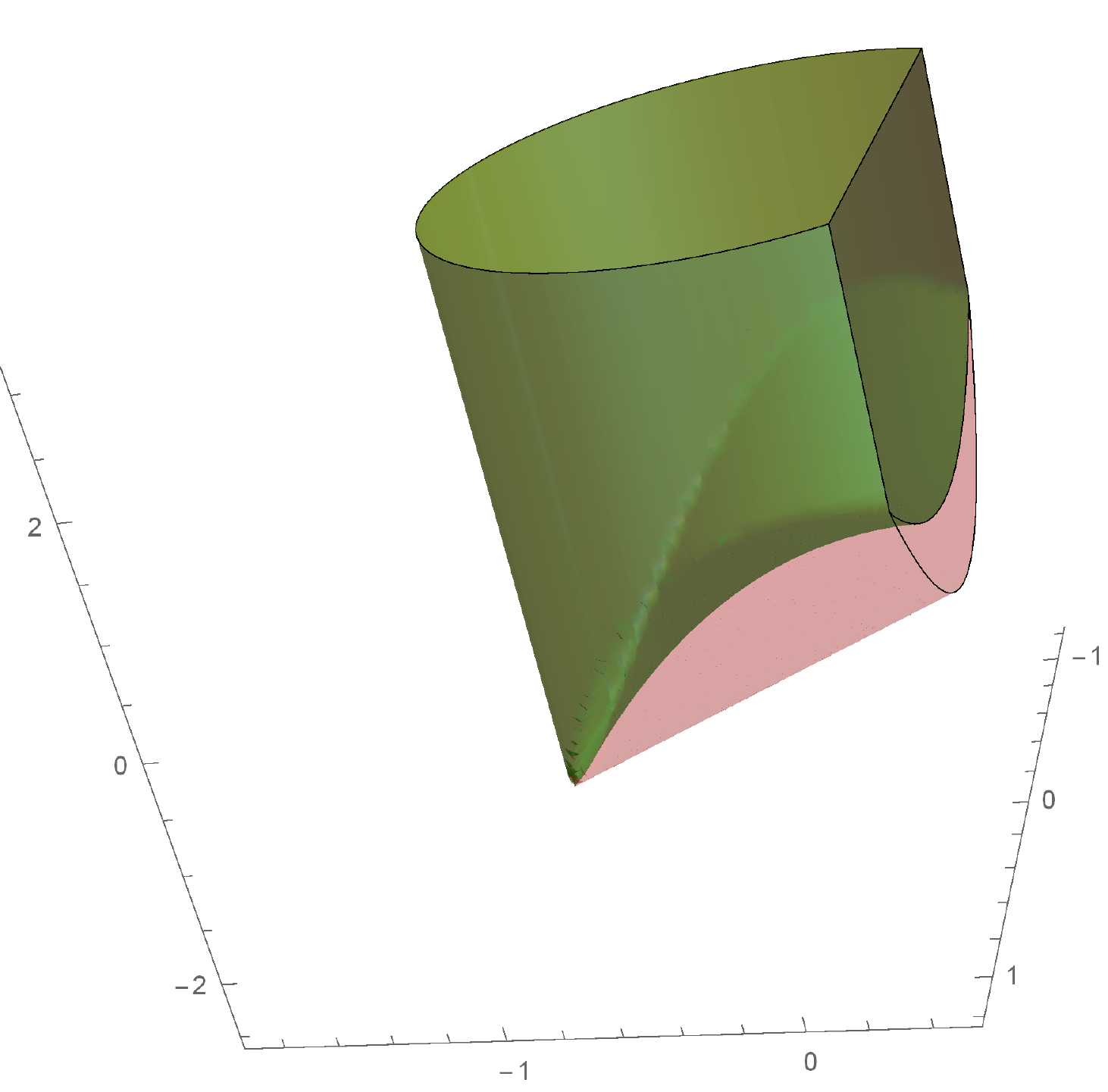}%
  }
\caption{Plots of the epigraph of Example~\ref{ex:condNecessityConvexHull}.}
\label{fig:ex2}
\end{figure}

Figure~\ref{ex2F0F1Fs} shows that the convex relaxation for the epigraph $X=\cF_0^+ \cap \cF_1 \cap \widehat{\cK} \cap H^1$ given by $\cF_0^+ \cap \cF_s \cap \widehat{\cK} \cap H^1$ does not give the convex hull of $X$. Note also that Condition~\ref{cond:TRS-relaxation} is satisfied for this example by taking $d = [0;1]$, and so by Theorem~\ref{thm:trsSOC-tight-convex}, the SOC optimization problem~\eqref{eqn:convex-trsSOC-conic} is a tight relaxation for \eqref{eqn:trsSOC-conic}. Despite this, we cannot give the exact convex hull characterization because Condition~\ref{cond:TRS-convexify} is violated.

If we were to set $b = 1$ instead of $b={1\over 2}$ in this example, then the linear inequality would become redundant. In this case, our convex relaxation would give the convex hull, as illustrated in Figure~\ref{fig:ex2-noIneq} below. Nevertheless, even in this case  Condition~\ref{cond:TRS-convexify} would still be violated. This demonstrates that Condition~\ref{cond:TRS-convexify} is not necessary to obtain the convex hull. 

\begin{figure}[htp]
\centering
  \subfigure[$X = \cF_0^+ \cap \cF_1 \cap \widehat{\cK} \cap H^1$]{%
    \label{ex2F0F1-noIneq}%
    \includegraphics[width=0.32\textwidth]{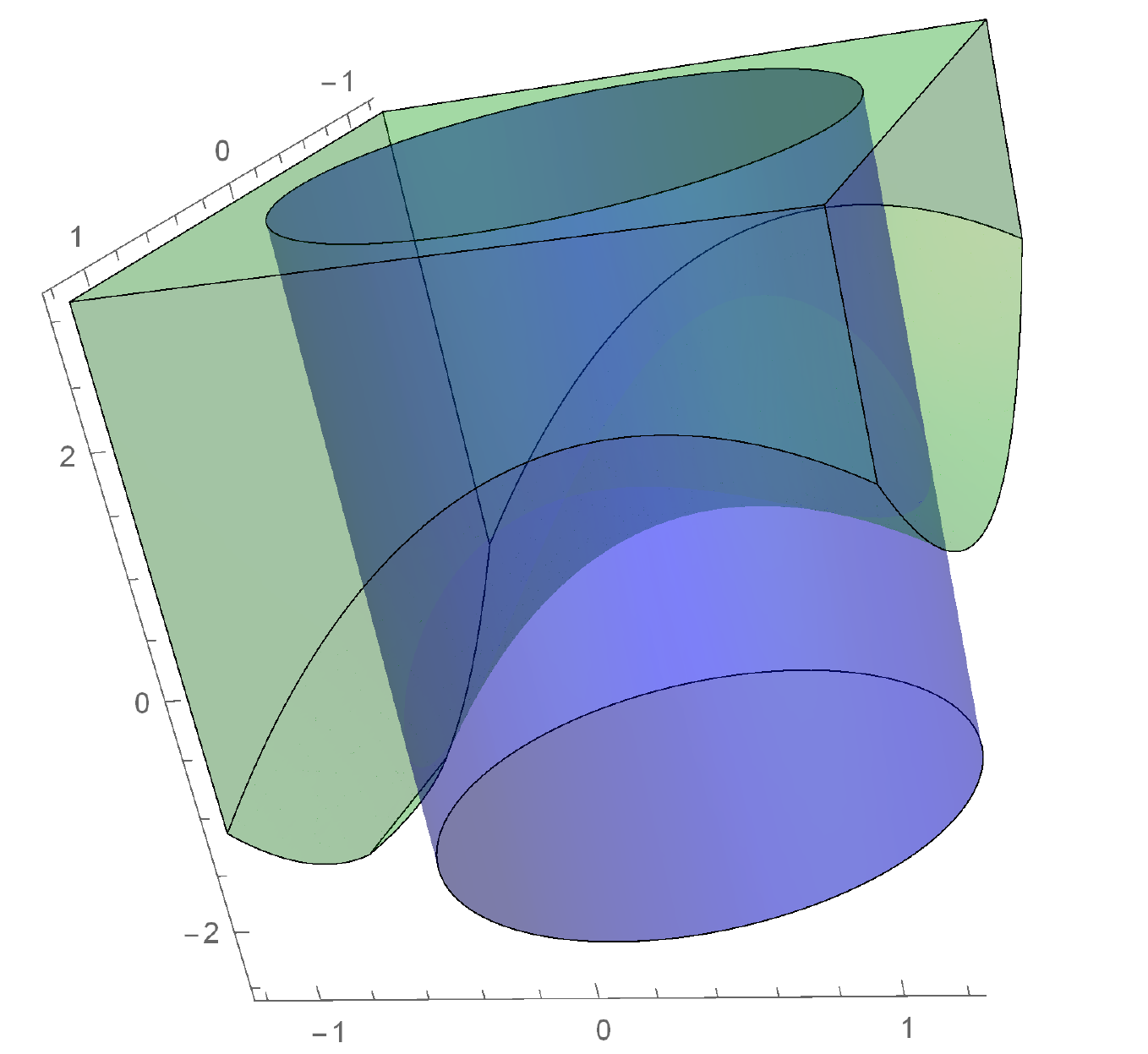}%
  } 
  \subfigure[$\cF_0^+ \cap \cF_s \cap \widehat{\cK} \cap H^1$]{%
    \label{ex2F0Fs-noIneq}%
    \includegraphics[width=0.32\textwidth]{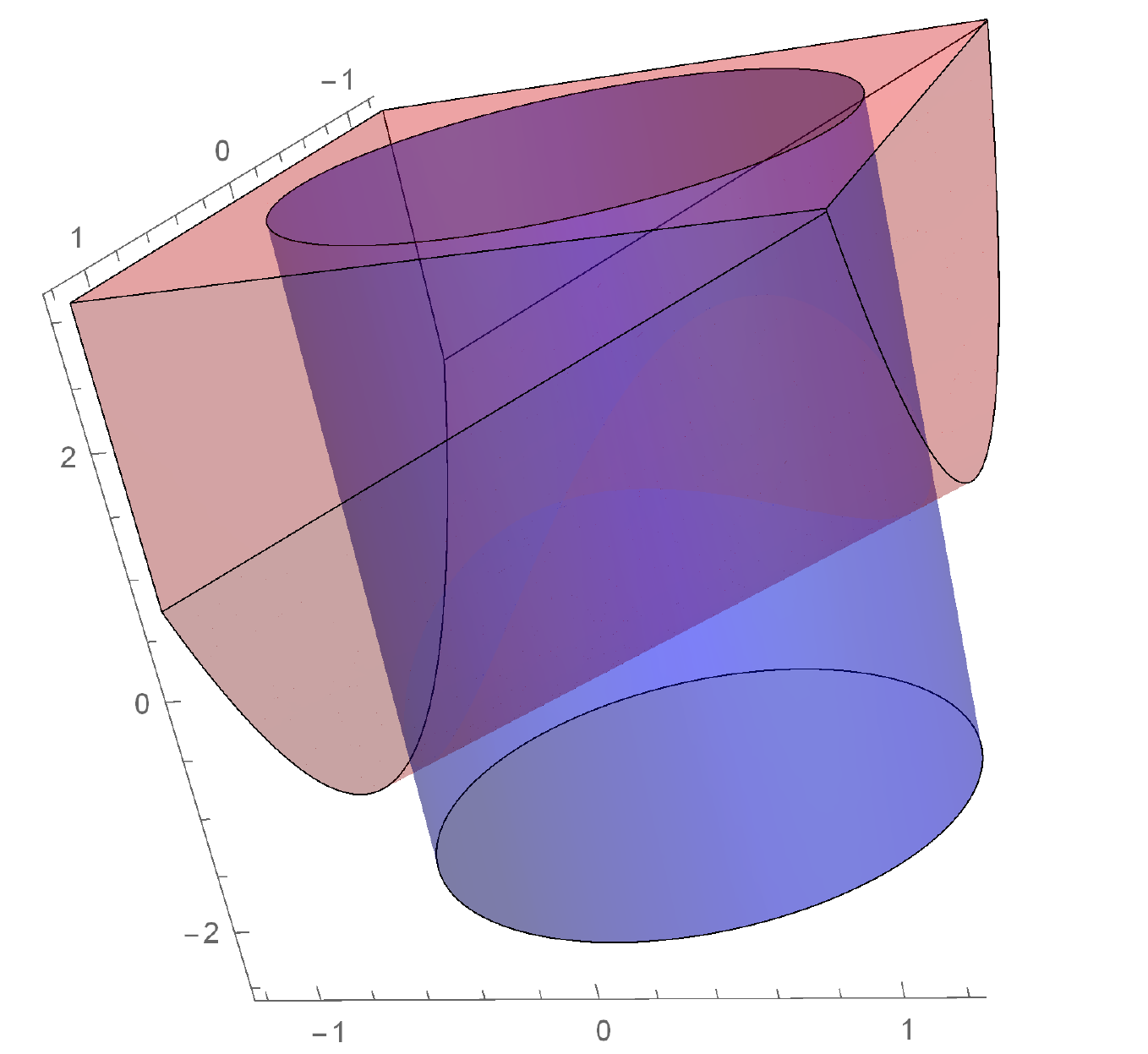}%
  }
  \subfigure[$X$ vs $\cF_0^+ \cap \cF_s \cap \widehat{\cK} \cap H^1$]{%
    \label{ex2F0F1Fs-noIneq}%
    \includegraphics[width=0.32\textwidth]{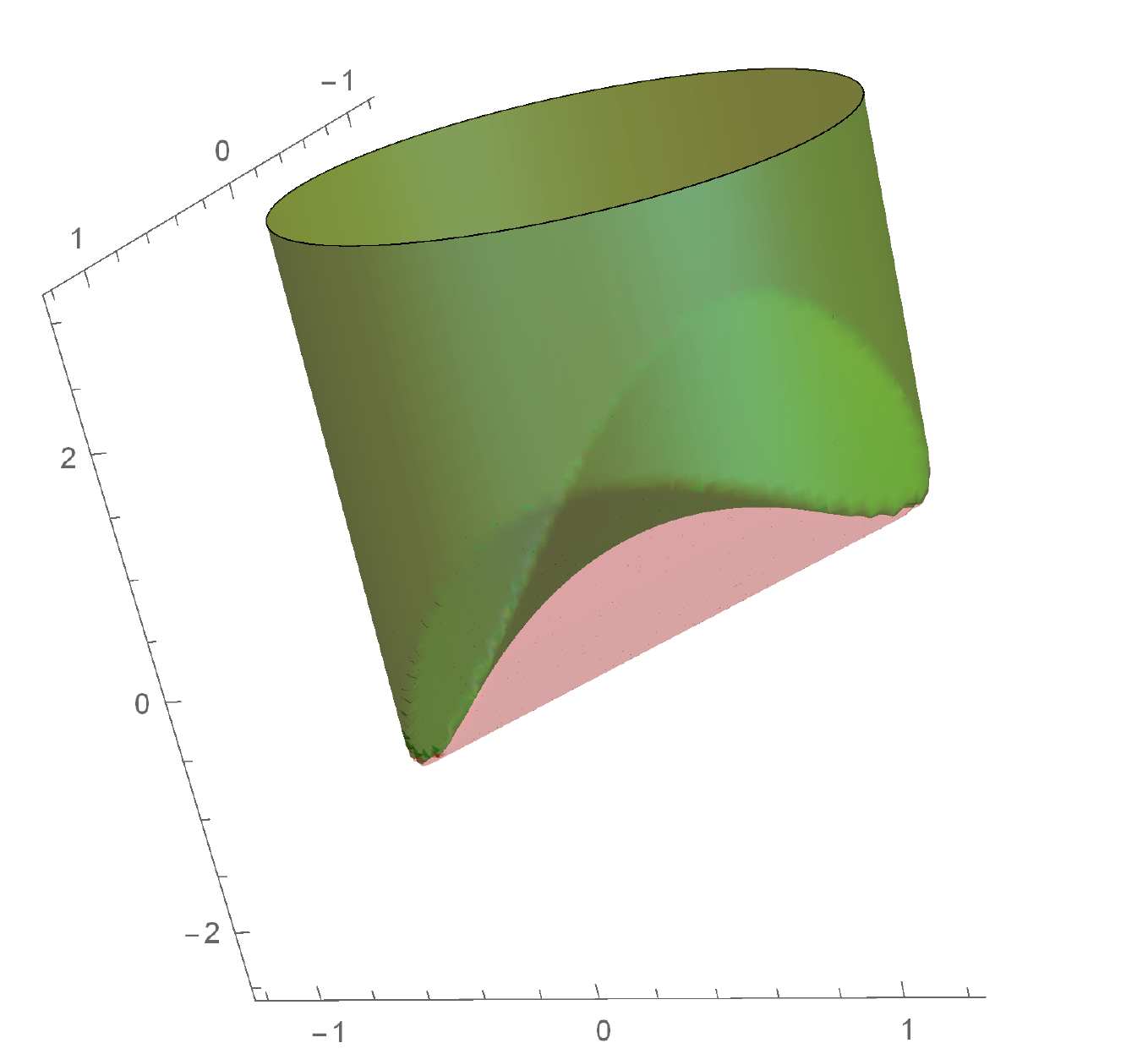}%
  }
\caption{Plots of the epigraph of Example~\ref{ex:condNecessityConvexHull} without the linear inequality.}
\label{fig:ex2-noIneq}
\end{figure}
\epr
\end{example}

We close this section with a simple result which highlights a particularly important structure of the extreme points of $X$.

\begin{lemma}\label{lemma:X-conv-comb}
Let $X$ be defined as in \eqref{eqn:X-set}. Assume that $\lambda_Q < 0$ and Condition~\ref{cond:TRS-convexify} also holds. 
Then any point $[y;1;x_{n+2}] \in X$ is an extreme point of $\conv(X)$ only if $\|y\| = 1$. 
\end{lemma}
\begin{proof}
Consider $[y;1;x_{n+2}] \in \conv(X)$ with $\|y\| < 1$. Let $d \neq 0$ be the vector given by Condition~\ref{cond:TRS-convexify}. Because $d$ satisfies $Qd = \lambda_Q d$ and $\pm Ad \in \cK$, for any $\epsilon \in \R$,
\begin{align*}
f(y+\epsilon d) &= (y+\epsilon d)^\top (Q - \lambda_Q I_n) (y+\epsilon d) + 2g^\top (y+\epsilon d) + \lambda_Q\\
&\quad= [y^\top (Q - \lambda_Q I_n) y + 2g^\top y + \lambda_Q ]+ 2g^\top d \epsilon\\
&\quad\leq x_{n+2} + 2g^\top d \epsilon.
\end{align*}
Now choose $\epsilon^+ > 0$ such that $\|y + \epsilon^+ d\| = 1$, and define $x^+ = [y+\epsilon^+ d; 1; x_{n+2} + 2g^\top d \epsilon^+]$. Then we have $x^+ \in X$, since $\|y+\epsilon^+ d\|=1$ guarantees $f(y+\epsilon^+ d) = h(y+\epsilon^+ d) \leq x_{n+2} + 2g^\top d \epsilon^+$. Note that we will have $\epsilon^+ > 0$ and finite since $\|y\| < 1$ and $d \neq 0$. Similarly, choosing $\epsilon^- > 0$ such that $\|y - \epsilon^+ d\| = 1$, $x^- = [y-\epsilon^- d; 1; x_{n+2} - 2g^\top d \epsilon^-] \in X $ also. Then the point $x$ will be a convex combination of $x^+, x^- \in X$ with weights $\epsilon^-/(\epsilon^- + \epsilon^+)$ and $\epsilon^+/(\epsilon^- + \epsilon^+)$ respectively.
\end{proof}

\subsection{Additional Hollow Constraints}\label{sec:TRS-additional-constraints}

In this section we explore additional constraints $y \in \cR$ included in the domain of TRS \eqref{eqn:trsSOC-conic}, where $\cR = \R^n\setminus \cP$ and $\cP$ is a given possibly nonconvex set. More precisely, we characterize the convex hull of the set $X \cap \widehat{\cR} = \cF_0^+ \cap \cF_1 \cap \widehat{\cK} \cap H^1 \cap \widehat{\cR}$ where $\cF_0^+$, $\cF_1$, $\widehat{\cK}$, and $H^1$ are as defined in \eqref{eqn:trs-cones-defn}, and $\widehat{\cR} := \left\{ [y;x_{n+1};x_{n+2}] : y \in \cR \right\}$.
 
We impose the following condition on $\cR = \R^n \setminus \cP$.
\begin{condition}\label{cond:TRS-hollows}
The set $\cP \subseteq \R^n$ satisfies $\cP \subseteq \left\{ y : \|y\| < 1,\ Ay-b \in \cK \right\}$.
\end{condition}
Consider the case where $Ay - b \in \cK$ is non-existent. If $\cP = \bigcup_{i=1}^m E_i$ is a union of ellipsoids $E_i = \left\{ y :\, y^\top W_i y + 2 b_i^\top y + c_i \leq 0 \right\}$ where each $W_i \succ 0$, then Condition~\ref{cond:TRS-hollows} can be checked by solving
\[
v_i = \min_y \left\{ 1 - \| y\|^2 :~ y^\top W_i y + 2 b_i^\top y + c_i \leq 0 \right\}.
\]
That is, $E_i$ satisfies Condition~\ref{cond:TRS-hollows} if and only if $v_i > 0$. The computation of $v_i$ as stated requires solving a nonconvex quadratic program, which is nothing but a classical TRS after an appropriate affine transformation of the variables is applied. Hence, our developments from Section~\ref{sec:SOCReformulation} give a tight SOC reformulation for it. In addition, the inhomogeneous $\cS$-lemma \cite[Proposition 3.5.2]{BenTalNemirovski2015LectureNotes} ensures that the associated semidefinite relaxation is tight.  Thus, Condition~\ref{cond:TRS-hollows} can be verified  efficiently when $\cP$ is a union of ellipsoids.

Hollow constraints have been studied in TRS literature under conditions similar to Condition~\ref{cond:TRS-hollows}. Most notably, the \emph{interval-bounded TRS} \cite{BenTalTeboulle1996,BM14,PongWolkowicz2014,SternWolkowicz1995,YeZhang2003} corresponds to the case when $\cR$ is a single lower-bounded quadratic constraint $y^\top D y \geq l$, where $D \succeq 0$.
The interval-bounded TRS is used to generate new steps in the context of the trust-region algorithm where minimum step lengths are enforced. In the case of interval-bounded TRS, Condition~\ref{cond:TRS-hollows} is automatically satisfied. It is shown in a number of these papers \cite{PongWolkowicz2014,YeZhang2003} that the natural SDP relaxation of interval-bounded TRS is tight. 
More recently, Yang et al.~\cite{YangBurer2016} showed the tightness of the SDP relaxation when the hollow set $\cP$ is the disjoint union of ellipsoids which do not intersect the boundary of the unit ball $\{y:\|y\|\leq 1\}$. As opposed to these results on tight SDP relaxations, Bienstock~\cite{Bienstock2016} has established that the general quadratically constrained quadratic programming problem
\[ \min_y \left\{ y^\top Q_0 y + 2g_0^\top y :~ y^\top Q_i y + 2 g_i^\top y + c_i \leq 0,\ i=1,\ldots,m \right\} \]
is polynomially solvable for a fixed number of constraints $m$ using a weak feasibility oracle, under the assumption that at least one quadratic constraint $y^\top Q_i y + 2 g_i^\top y + c_i \leq 0$ is strictly convex. In a similar vein, Bienstock and Michalka \cite{BienstockMichalka2014} also study TRS with additional ellipsoidal hollow constraints. Instead of giving the convex hull, \cite{BienstockMichalka2014} explores conditions that allow for polynomial solvability using a combinatorial enumeration technique and thus is able to cover cases where the set $\cP$ may not be contained in the unit ball. On a related subject, \cite{BM14} studies the characterization and separation of valid linear inequalities that convexify the epigraph of a convex, differentiable function whose domain is restricted to the complement of a convex set defined by linear or convex quadratic inequalities.

We note that these papers \cite{BenTalTeboulle1996,Bienstock2016,BM14,BienstockMichalka2014,PongWolkowicz2014,SternWolkowicz1995,YangBurer2016,YeZhang2003} consider the more general case of minimizing an arbitrary quadratic objective, which can be convex, over a domain given by possibly nonconvex quadratic constraints.  
On the other hand, our result applies to the special case of minimizing a nonconvex quadratic, i.e., $\lambda_Q < 0$, over the unit ball, a convex quadratic constraint. As a result, we are able to relax the assumptions that the set $\cP$ is generated by quadratics and the ellipsoidal hollows are disjoint.
Specifically, we show that under Condition~\ref{cond:TRS-hollows}, our  main convex hull result, i.e.,Theorem~\ref{thm:trs-convexify}, obtained without the constraint $y\in\cR$ is tight. 

\begin{theorem}\label{thm:TRS-convexify-hollows}
Let $X$ be defined in \eqref{eqn:X-set} and $\cR = \R^n \setminus \cP$ be a set satisfying Condition~\ref{cond:TRS-hollows}. Assume that $\lambda_Q < 0$ and Condition~\ref{cond:TRS-convexify} also holds. Then
\[
\conv\left(\left\{ [y;1;x_{n+2}] : \begin{array}{rcl} \|y\| &\leq& 1\\ y &\in& \cR\\ Ay - b &\in& \cK\\ y^\top Q y + 2g^\top y &\leq& x_{n+2} \end{array} \right\}\right) = \conv(X).
\]
\end{theorem}
\begin{proof}
Denoting $\widehat{\cR} := \left\{ [y;x_{n+1};x_{n+2}] : y \in \cR \right\}$, our aim is to prove $\conv(X \cap \widehat{\cR}) = \conv(X)$. We trivially have $\conv(X \cap \widehat{\cR}) \subseteq \conv(X)$. 
To prove $\conv(X \cap \widehat{\cR}) \supseteq \conv(X)$, note that from Lemma~\ref{lemma:X-conv-comb}, any point $x = [y;1;x_{n+2}] \in \Ext(\conv(X))$ satisfies $\|y\|=1$. Also, by Condition~\ref{cond:TRS-hollows}, the constraint $x\in\widehat{R}$ does not remove any of the points with $\|y\|=1$, in particular, all of the extreme points of $X$ are also in $\widehat{R}$. Thus, $\Ext(X\cap\widehat{R})=\Ext(X)$. Moreover, because $\|y\|\leq 1$, the only recessive direction of $\conv(X)$ is $[0;0;1]$, i.e., $\Rec(\conv(X))=\cone([0;0;1])$. Note $[0;0;1]$ is also a recessive direction in $\widehat{R}$. Then the result follows from 
\begin{align*}
\conv(X\cap\widehat{R}) &= \conv(\Ext(X\cap\widehat{R}))+\Rec(X\cap\widehat{R}) \\
&= \conv(\Ext(X))+\Rec(X)=\conv(X) .
\end{align*}

\end{proof}

Theorem~\ref{thm:TRS-convexify-hollows} has the following immediate implication. 
\begin{corollary}\label{cor:TRS-convexify-hollows}
When $\lambda_Q < 0$ and $l \leq 1$, an exact convex reformulation of the interval-bounded TRS 
\[ \min_y \left\{ y^\top Q y + 2g^\top y :~ l \leq \|y\| \leq 1 \right\} \]
is given by \eqref{eqn:convex-trsSOC-classical}.
\end{corollary}

Corollary~\ref{cor:TRS-convexify-hollows} gives a convex reformulation for the interval-bounded TRS with $\lambda_Q<0$, as do results from \cite{BenTalTeboulle1996,PongWolkowicz2014,SternWolkowicz1995,YeZhang2003}. These results were often derived as a consequence of a simultaneously diagonalizable assumption of the underlying matrices associated with TRS, or through SDP relaxations. In contrast to this, Condition~\ref{cond:TRS-hollows} and Theorem~\ref{thm:TRS-convexify-hollows} highlights the important geometric aspect, and provide a convex reformulation without additional variables. In addition, Corollary~\ref{cor:TRS-convexify-hollows} together with Theorem~\ref{thm:classicalTRScomplexity} establish the convergence rate of FOMs to solve interval-bounded TRS as opposed to specialized algorithms suggested in \cite{PongWolkowicz2014}.

\begin{remark} 
While preparing our first revision, an overlap of our original submission with a recent paper was brought to our attention. 
The first version of our work \cite{Ho-NguyenKK2016TRSWeb} was published online on March 10, 2016 on archives \emph{Optimization Online} and \emph{arXiv} and sent to a journal for possible publication. Five months after this date, on August 20, 2016, a paper by Jiulin Wang and Yong Xia \cite{WangXia2016} has appeared in online form on the journal \emph{Optimization Letters}. It appears that this paper was submitted to \emph{Optimization Letters} on March 21, 2016, 11 days after our paper was posted in public domain. To the best of our knowledge, Wang and Xia's paper \cite{WangXia2016} was not available in public domain or available to us before August 20, 2016. The paper \cite{WangXia2016} has significant overlap with a part of the results in our paper. In particular, \cite[Theorem 1]{WangXia2016} is a corollary of results in our original submission, see \cite[Theorem 2.7 and Theorem 3.8]{Ho-NguyenKK2016TRSWeb}. Moreover, we were the first ones to note and discuss the use of Nesterov's accelerated gradient descent algorithm in the context of TRS and demonstrate that it achieves the best-known theoretical convergence rate to solve TRS and some of its variants. Specifically, \cite[Section 2.2]{Ho-NguyenKK2016TRSWeb} along with the conclusion of \cite[Theorem 3.8]{Ho-NguyenKK2016TRSWeb} in the case of interval-bounded TRS covers not only \cite[Section 3]{WangXia2016} but also highlights that there is no need to modify Nesterov's accelerated gradient descent algorithm to solve exact convex reformulation of the interval-bounded TRS. Therefore, the convergence rate established in \cite[Section 2.2]{Ho-NguyenKK2016TRSWeb} did already imply \cite[Theorem 4]{WangXia2016}. To the best of our understanding, the main results in  \cite{WangXia2016} are \cite[Theorems 1 and 4]{WangXia2016}; whereas we also study SOC-based convex reformulations and convex hull descriptions of TRS and its variants with additional conic constraints and/or general hollow constraints see \cite[Sections 2.1 and 3]{Ho-NguyenKK2016TRSWeb}.

\epr
\end{remark}

\section*{Acknowledgments}
The authors wish to thank the review team for their constructive feedback that improved the presentation of the material in this paper. 
This research is supported in part by NSF grant CMMI 1454548.

\bibliographystyle{abbrv}
\bibliography{bibliography}

\appendix
\section{Working with Approximate Eigenvalues}\label{app:approx-eigenvalues}
Consider the classical TRS \eqref{eqn:trsSOC-classical} and its convex reformulation \eqref{eqn:convex-trsSOC-classical}. In practice, we will actually form the objective $y^\top (Q - \gamma I_n) y + 2g^\top y + \gamma$ where $\gamma \approx \lambda_Q$ is an approximation. Due to this imprecision, we must ensure that the objective remains convex. To do this, suppose that we solve the minimum eigenvalue problem of $Q$ to within $\epsilon$-accuracy, and obtain  an approximate solution $\lambda_Q-\epsilon < \lambda < \lambda_Q + \epsilon$. Subtracting $\epsilon$ from the inequality, we obtain $\lambda_Q - 2\epsilon < \lambda - \epsilon < \lambda_Q$. To ensure the convexity of the objective, we set $\gamma := \lambda - \epsilon < \lambda_Q$ which is an underestimate of $\lambda_Q$, ensuring that $Q - \gamma I_n \succ 0$. Let $\eta := \lambda_Q - \gamma$ which satisfies $0 < \eta < 2\epsilon$, and 
\[
f_\eta(y) := y^\top (Q-\gamma I_n) y + 2g^\top y = y^\top (Q-(\lambda_Q - \eta) I_n) y + 2g^\top y = f(y) + \eta \|y\|^2.
\]
Based on this scheme, we next explore the effects of solving
\begin{equation}\label{eqn:TRS-approx-objective}
\min_y \left\{ f_\eta(y) :~ \|y\| \leq 1 \right\}
\end{equation}
instead of 
\eqref{eqn:convex-trsSOC-classical}. Let $y^*$ be an optimal solution to the true convex reformulation \eqref{eqn:convex-trsSOC-classical}. Let $y^\eta$ be an optimal solution to \eqref{eqn:TRS-approx-objective} and $\bar{y}^\eta$ be an approximate optimal solution. Then we can bound the objective value $f(\bar{y}^\eta)$ as
\[
f(\bar{y}^\eta) - f(y^*) = f_\eta(\bar{y}^\eta) - f_\eta(y^*) + \eta(\|y^*\|^2 - \|\bar{y}^\eta\|^2) \leq f_\eta(\bar{y}^\eta) - f_\eta(y^\eta) + \eta,
\]
where the last inequality follows from $\|y^*\| \leq 1$ and $\|\bar{y}^\eta\|\leq 1$.
Thus, the convergence rate of $\bar{y}^\eta$ to the optimum of \eqref{eqn:convex-trsSOC-classical} is controlled by the size of $\eta$ and the convergence rate for solving \eqref{eqn:TRS-approx-objective}.

We can also control the distance between $y^\eta$ and $y^*$. Because $f_\eta(y)$ is a $2\eta$-strongly convex function, we have
\begin{align*}
\eta \left\| y^* - y^\eta \right\|^2 &\leq f_\eta(y^*) - f_{\eta}(y^\eta) + \grad f_{\eta}(y^\eta)^\top (y^\eta - y^*)\\
&= f(y^*) - f(y^\eta) + \grad f_{\eta}(y^\eta)^\top (y^\eta - y^*) + \eta(\|y^*\|^2 - \|y^\eta\|^2)\\
&\leq \eta(\|y^*\|^2 - \|y^\eta\|^2),
\end{align*}
where the last inequality follows from the optimality of $y^\eta$ for the problem~\eqref{eqn:TRS-approx-objective}, i.e., $\grad f_{\eta}(y^\eta)^\top (y^\eta - y^*) \leq 0$, and the optimality of $y^*$ for the problem~\eqref{eqn:convex-trsSOC-classical}. Then $\|y^\eta\| \leq \|y^*\|$. Also, from $\|y^*\| \leq 1$, we deduce that if $\|y^\eta\| = 1$, then $y^* = y^\eta$. When $\|y^\eta\| < 1$, the only constraint in our domain is inactive, and thus we conclude that $y^\eta$ is also optimum for the unconstrained minimization problem. Then  the optimality conditions leads to $\grad f_\eta(y^\eta) = 0$. This implies that $y^\eta = -(Q + (\eta - \lambda_Q) I_n )^{-1} g$. Moreover, $y^*$ satisfies the optimality condition $\grad f(y^*)^\top (y^* - y) \leq 0$ for all $y$ such that $\|y\| \leq 1$. Since our domain is the unit ball, this is true if and only if $\grad f(y^*) = -\alpha y^*$, for some $\alpha \geq 0$. Therefore, $y^* = -(Q + (\alpha - \lambda_Q) I_n )^{\dagger} g$, where $A^{\dagger}$ denotes the pseudo-inverse of a matrix $A$. If we denote the ordered eigenvalues of $Q$ by $q_i$ and their corresponding orthonormal eigenvectors by $u_i$, we obtain
\[
\|y^\eta\|^2 = \sum_{i=1}^{n} \frac{(u_i^\top g)^2}{(q_i - q_n + \eta)^2}\quad \mbox{and}\quad
\|y^*\|^2 = \sum_{i=1}^{n} \frac{(u_i^\top g)^2}{(q_i - q_n + \alpha)^2}.
\]
Note that it is possible to have $\alpha = 0$ and $q_i - q_n = 0$. However, this happens only when $u_i^\top g = 0$, so we follow the convention $\frac{0}{0} = 0$. After some simple algebra, we have the equality
\begin{align*}
\|y^*\|^2 - \|y^\eta\|^2 &= \sum_{i=1}^{n} \frac{(u_i^\top g)^2}{(q_i - q_n + \alpha)^2} - \sum_{i=1}^{n} \frac{(u_i^\top g)^2}{(q_i - q_n + \eta)^2}\\
&= (\eta - \alpha) \sum_{i=1}^{n} (u_i^\top g)^2 \frac{2q_i - 2q_n + \eta + \alpha}{(q_i - q_n + \alpha)^2 (q_i - q_n + \eta)^2}.
\end{align*}
Since $\|y^*\| \geq \|y^\eta\|$ and $\eta>0$, we must have $\eta \geq \alpha$. Also, $\eta \leq \alpha$ is possible only if $y^\eta = y^*$. Hence, we have
\begin{align*}
\|y^*\|^2 - \|y^\eta\|^2 &= (\eta - \alpha)_+ \sum_{i=1}^{n} (u_i^\top g)^2 \frac{2q_i - 2q_n + (\eta - \alpha)_+ + 2\alpha}{(q_i - q_n + \alpha)^2 (q_i - q_n + (\eta-\alpha)_+ + \alpha)^2}\\
&\leq (\eta - \alpha)_+ \sum_{i=1}^{n} (u_i^\top g)^2 \frac{2q_i - 2q_n + (\eta - \alpha)_+ + 2\alpha}{(q_i - q_n + \alpha)^4}\\
&= 2(\eta - \alpha)_+ \sum_{i=1}^{n} \frac{(u_i^\top g)^2}{(q_i - q_n + \alpha)^3} + (\eta - \alpha)_+^2 \sum_{i=1}^{n} \frac{(u_i^\top g)^2}{(q_i - q_n + \alpha)^4}.
\end{align*}
This shows that $\|y^*\|^2 - \|y^\eta\|^2 \leq \phi \eta + o(\eta)$, where $\phi = 2(y^*)^\top (Q+(\alpha - \lambda_Q)I_n)^{\dagger} y^*$. Therefore, 
\[
\|y^\eta - y^*\|^2 \leq \|y^*\|^2 - \|y^\eta\|^2 \leq \phi \eta + o(\eta).
\]
Thus $y^\eta$ has error $O(\sqrt{\eta})$, which is expected since the error in the objective function is $O(\eta)$, and the objective function is quadratic.

\section{Computation of $s$ value}\label{app:s-computation}
Recall the notation $\tilde{y} = [y;\tilde{y}_{n+1}]$ and $\tilde{x} = [\tilde{y};x_{n+1};x_{n+2}]$. For the set $Y$ in \eqref{eqn:BKK-nonconvex-set}, Condition \ref{cond:one_neg_eval} is satisfied by construction, and Condition \ref{cond:interior} is satisfied by taking $\tilde{x}'=[y'; \tilde{y}_{n+1}';x_{n+1}';x_{n+2}']$ with $y'=0$, $\tilde{y}_{n+1}'=\frac{1}{2}$, $x_{n+1}'=1$ and $x_{n+2}'=0$. 
This ensures that for any $t \in [0,1]$, we have
\begin{equation}\label{eqn:BKK-matrices-t}
\tilde{W}_t = (1-t) \tilde{W}_0 + t\tilde{W}_1 = \begin{bmatrix} (1-t)I_{n+1} + t \tilde{Q} & t \tilde{g} & 0 \\ t \tilde{g}^\top  & t-1 & 0\\ 0^\top  & 0 & t \end{bmatrix},
\end{equation}
and $(\tilde{x}')^\top \tilde{W}_t \tilde{x}' = (\tilde{x}')^\top  ((1-t)\tilde{W}_0 + t\tilde{W}_1) \tilde{x}' < 0$. Thus, by the variational characterization of eigenvalues, $\tilde{W}_t$ has at least one negative eigenvalue. 
Also, Condition \ref{cond:A0_apex}(ii) is now satisfied. 

We next show that the precise value of $s$ is simply determined by $\lambda_Q$. 

\begin{lemma}\label{lem:s-value}
Suppose $\lambda_Q<0$. Consider $\tilde{W}_0,\tilde{W}_1$ as defined in \eqref{eqn:BKK-matrices}. Then, the maximal $t\in[0,1]$ that ensures that the matrix $\tilde{W}_t$ in \eqref{eqn:BKK-matrices-t} has a single negative eigenvalue for all $t\in[0,s]$, is invertible for all $t\in(0,s)$, 
and $\tilde{W}_s$ is singular is given by  
\[
s = \frac{1}{1-\lambda_Q} \in (0,1).
\]
\end{lemma}
\begin{proof}
Define $\hat{s} := \frac{1}{1-\lambda_Q} \in (0,1)$. From \eqref{eqn:BKK-matrices-t}, note that $\tilde{W}_t$ has a block structure and $s$ is such that it equals to the smallest positive $t$ ensuring
\[
  V_t := (1-t) \begin{bmatrix} I_{n+1} & 0 \\ 0 & -1 \end{bmatrix}
  + t \begin{bmatrix} \tilde{Q} & \tilde{g} \\ \tilde{g}^\top  & 0 \end{bmatrix}
\]
is singular. 

Let $\lambda_{n+2,t},\lambda_{n+1,t}$ be the two smallest eigenvalues of $V_t$, and  $\rho_{n+1,t},\rho_{n,t}$ be the two smallest eigenvalues of $(1-t)I_{n+1} + t \tilde{Q}$. Notice that $(1-t)I_{n+1} + t \tilde{Q}$ has the same eigenvectors as $\tilde{Q}$, and the eigenvalues are simply scaled and shifted from those of $Q$, thus the minimum eigenvalue of $(1-t)I_{n+1} + t \tilde{Q}$ is $1-t + t\lambda_Q$ for $t \in (0,1)$. Also, by construction, the multiplicity of $\lambda_Q$ in $\tilde{Q}$ is at least two, so the multiplicity of the minimum eigenvalue of $(1-t)I_{n+1} + t \tilde{Q}$ is also at least two, therefore $\rho_{n+1,t} = \rho_{n,t} = 1-t + t\lambda_Q$.

For any $t \in (0,1)$, the last diagonal entry of $V_t$ is negative implying $V_t$ is not positive semidefinite, hence $\lambda_{n+2,t} < 0$. However, for $t \in (0,\hat{s})$, $\rho_{n+1,t} > 0$, and from Cauchy's interlacing theorem for eigenvalues \cite[Theorem 4.3.17]{Horn_Johnson_2013},  
we obtain 
\[ \lambda_{n+2,t} < 0 < \rho_{n+1,t} = \lambda_{n+1,t} = \rho_{n,t}, \quad t \in (0,\hat{s}).\]
Thus, for any $t \in (0,\hat{s})$, the matrix $V_t$, and hence $\hat{W}_t$, is invertible, and  $\hat{W}_t$ has exactly one negative eigenvalue. When $t=\hat{s}$, $\rho_{n+1,\hat{s}} = \rho_{n,\hat{s}} = 1-\hat{s} + \hat{s} \lambda_Q = 0$. By recalling that $\tilde{Q}:=\begin{bmatrix} Q & 0\\ 0 &\lambda_Q\end{bmatrix}$ and $\tilde{g} = [g;0]$, we immediately observe that $V_{\hat{s}}$, and thus $\tilde{W}_{\hat{s}}$, is singular since $V_{\hat{s}}$ has eigenvector $[y;\tilde{y}_{n+1};x_{n+1}] = [0;1;0]$ with eigenvalue $0$. Also,
\[\lambda_{n+2,\hat{s}} < 0 = \rho_{n+1,\hat{s}} = \lambda_{n+1,\hat{s}}  = \rho_{n,\hat{s}}\]
so $\tilde{W}_{\hat{s}}$ has exactly one negative eigenvalue. 
Moreover, for any $t>\hat{s}$, the minimum eigenvalue of $(1-t)I_n + t Q$ is $1-t + t\lambda_Q < 0$. Hence, for any $t>\hat{s}$, $\lambda_{n+2,t} \leq \rho_{n+1,t} = \lambda_{n+1,t} = \rho_{n,t}<0$ follows from \cite[Theorem 4.3.17]{Horn_Johnson_2013}. As a result $V_t$, and thus $\hat{W}_t$, has at least two negative eigenvalues. Therefore, $s=\hat{s} = \frac{1}{1-\lambda_Q}$ is the correct value.
\end{proof}

Choosing $\tilde{x}''=[y''; \tilde{y}_{n+1}'';x_{n+1}'';x_{n+2}'']$ with $y''=0$, $\tilde{y}_{n+1}''=1$, $x_{n+1}''=0$ and $x_{n+2}''=0$ ensures that $\tilde{x}'' \in \Null(\tilde{W}_s)$, $(\tilde{x}'')^\top  \tilde{W}_1 \tilde{x}'' < 0$, and $x_{n+1}''=0$. This simultaneously verifies Conditions \ref{cond:As_null} and \ref{cond:hyperplane}.

\end{document}